\documentclass[12pt]{article}
\usepackage{latexsym,bm}
\usepackage{amsfonts,graphicx,amssymb,amsmath,graphicx,amsthm,comment}
\usepackage{mathrsfs,mathrsfs,dsfont,galois,booktabs,ctable,threeparttable,tabularx,algorithm,algorithmic}
\usepackage[english,german]{babel}
\usepackage[T1]{fontenc}
\usepackage[latin1]{inputenc}
\usepackage[colorlinks=true]{hyperref}
\usepackage{color}

\newtheorem{Theorem}{Theorem}[section]
\newtheorem{Lemma}{Lemma}[section]
\newtheorem{Remark}{Remark}[section]
\newtheorem{Proposition}{Propositon}[section]
\newtheorem{Definition}{Definition}[section]
\newtheorem{Corollary}{Corollary}[section]
\numberwithin{equation}{section}
\hypersetup{linkcolor=blue,urlcolor=red,citecolor=red}

\begin{document}
\selectlanguage{english}

\title{Attainable subspaces and the bang-bang property of time optimal controls for heat equations}

\author {
G. Wang\thanks{
 School of Mathematics and Statistics, Wuhan University, Wuhan, 430072, China (wanggs62@
yeah.net) The author was partially supported by NNSF of China under
grant 11171264 and the National Basis Research Program of China
under grant 2011CB808002.} ,\;\; Y. Xu\thanks{School of Mathematical
Sciences, Fudan University, KLMNS, Shanghai, 200433, China
(yashanxu@fudan.edu.cn).}\;\; and\;\; Y. Zhang\thanks{School of
Mathematics and Statistics, Wuhan University, Wuhan, 430072, China
(yubiao\b{ }zhang@whu.edu.cn).} }

\date{}
\maketitle

\begin{abstract} In this paper, we  study two subjects on internally  controlled heat equations
with time varying potentials: the  attainable subspaces and the
bang-bang property for some time optimal control problems. We
present some equivalent characterizations on the  attainable
subspaces, and provide a sufficient conditions to ensure the
bang-bang property. Both the above-mentioned characterizations and
the sufficient condition are closely related to  some function
spaces consisting of some solutions to the adjoint  equations. It
seems for us that   the existing ways  to derive the bang-bang
property for heat equations with time-invariant potentials (see, for
instance, \cite{HOF}, \cite{GL}, \cite{MS} and \cite{Wang}) do not
work for the case where the potentials are time-varying. We provide
another way to approach it in the current paper.\\

\noindent\textbf{2010 AMS Subject Classifications. 49K20, 93C20}  \\

\noindent\textbf{Keywords.} Attainable subspaces, the bang-bang
property, time optimal controls, heat equations

\end{abstract}

\section{Introduction}
Let $\Omega\subset\mathbb R^d$, $d\geq 1$, be a bounded domain with a $C^2$ boundary $\partial \Omega$. Write $\omega\subset \Omega$ for
an open and non-empty subset with its characteristic function $\chi_\omega$. Consider the controlled heat equation:
\begin{equation}\label{originalequation}
\begin{cases}
y_t-\Delta y+a y=\chi_\omega u & \textrm{ in }\Omega\times \mathbb R^+,\\
y=0                            &\textrm{ on }\partial\Omega\times\mathbb R^+,\\
y(x,0)=y_0(x)                  &\textrm{ in }\Omega,
\end{cases}
\end{equation}
where $a\in L^\infty(\Omega\times\mathbb R^+)$, $y_0\in L^2(\Omega)$
and $u\in L^p(\mathbb R^+;L^2(\Omega))$, with $1<p\leq \infty$. We
will treat the solution
 of Equation (\ref{originalequation}) as a function from
 $\mathbb{R}^+$ to $L^2(\Omega)$, and denote it by $y(\cdot; y_0,u)$.
When  $\hat u\in L^p(0,T;L^2(\Omega))$ for some $T>0$, we use
$y(\cdot; y_0,\hat u)$ to stand for  the solution of Equation
(\ref{originalequation}), where $u=\hat u$ over $(0,T)$ and $u=0$
over $(T, \infty)$. Throughout the paper, $\|\cdot\|$ and
$\langle\cdot,\cdot\rangle$ stand for the usual norm and inner
product in $L^2(\Omega)$;
 $\|\cdot\|_\omega$ and $\langle\cdot,\cdot\rangle_\omega$ denote the usual norm and inner product in $L^2(\omega)$.
Given $T>0$ and $z\in L^2(\Omega)$, write  $\varphi(\cdot;T,z)$ for the solution to the
adjoint equation:
\begin{equation}\label{adjoint-equation}
\begin{cases}
\varphi_t+\Delta\varphi-a\varphi=0&\textrm{ in }\Omega\times(0,T),\\
\varphi=0&\textrm{ on }\partial\Omega\times(0,T),
\end{cases}
\end{equation}
with the initial condition $\varphi(T)=z$ over $\Omega$.

This paper studies two subjects on internally controlled equation
(\ref{originalequation}):  the attainable subspaces and the
bang-bang property of some time optimal control problems. These
subjects  are related to the spaces $Y_{T, q}$ (with $T>0$ and
$1\leq q<\infty$), which are defined by
\begin{equation}\label{wang1.8}
 Y_{T,q}=\overline{X_{T,q}}^{\|\cdot\|_{L^q(0,T;L^2(\omega))}},
\end{equation}
endowed with the norm:
\begin{equation}\label{WANG1.6}
\|\cdot\|_{Y_{T,q}}\triangleq\|\cdot\|_{L^q(0,T;L^2(\omega))},
\end{equation}
 where $  X_{T,q}=\big\{\chi_\omega\varphi(\cdot;T,z)|z\in L^2(\Omega)\big\}$ endowed with the $L^q(0,T;L^2(\omega))$-norm.
  We start with introducing the attainable subspaces. The attainable subspaces of (\ref{originalequation}) at time $T>0$ are defined by
\begin{equation}\label{wang1.5}
 A_{T,p} = \big\{y(T; 0, u)\,~\big|~ u\in L^p(0,T;L^2(\Omega))
 \big\},\; 1<p\leq \infty,
\end{equation}
endowed with the norms:
\begin{equation}\label{jiaxu1.6}
\|y_T\|_{A_{T,p}}\triangleq\inf \big\{
\|u\|_{L^p(0,T;L^2(\Omega))}~\big|~ y(T; 0, u)=y_T\big\},\;\;y_T\in
A_{T,p}.
\end{equation}
 We next introduce
 the following time optimal control problem $(TP)_{y_0}^{M,p}$:
\begin{equation}\label{TP}
 T_p(M,y_0)\triangleq
\inf_{u\in\mathcal{U}^{M,p}}\big\{t\;\big| \; y(t;y_0,u)=0 \big\},
\end{equation}
where $y_0\in L^2(\Omega)\backslash\{0\}$, $M>0$, $1<p\leq \infty$
and
$$
\mathcal{U}^{M,p}\triangleq \big\{v\,:\, \mathbb R^+\to
L^2(\Omega) \mid ~\|v\|_{L^p(\mathbb R^+;L^2(\Omega))}\leq M \big\}.
$$
In Problem $(TP)_{y_0}^{M,p}$, $u^*\in \mathcal{U}^{M,p}$ is called
an optimal control if $y(T_p(M,y_0); y_0, u^*)=0$; while $\hat u\in
\mathcal{U}^{M,p}$ is called an admissible control if $y(T; y_0,
\hat u)=0$ for some $T>0$.
\begin{Definition}\label{definition 1.1}
 Problem $(TP)_{y_0}^{M,p}$ has the
bang-bang property if any optimal control $u^*$ verifies that
$\|\chi_\omega u^*\|_{L^p( 0, T_p(M,y_0); L^2(\Omega))}=M$ and
$\|\chi_\omega u^*(t)\|\neq 0$ for a.e. $t\in (0,T_p(M,y_0))$, when
$1<p<\infty$; while $\|\chi_\omega u^*(t)\|=M$ for a.e. $t\in
(0,T_\infty(M,y_0))$,
 when $p=\infty$.
\end{Definition}

\begin{Remark}\label{remarkwang1.1}
We agree that when $(TP)_{y_0}^{M,p}$ has no any optimal control, it
does not hold the bang-bang property.

\end{Remark}

\noindent Our studies on   $(TP)^{M,p}_{y_0}$ are connected with the
norm optimal control problem $(NP)_{y_0}^{T,p}$:
\begin{equation}\label{NP}
N_p(T,y_0)\triangleq\inf\big\{\|u\|_{L^p(0,T;L^2(\Omega))}\,\big|\,y(T;y_0,u)=0\big\},\;\;T>0.
\end{equation}
In Problem $(NP)_{y_0}^{T,p}$, $u^*$ is called an optimal control if
$\|u^*\|_{L^p(0,T;L^2(\Omega))}=N_p(T,y_0)$ and $y(T; y_0, u^*)=0$;
while $\hat u\in L^p(0,T;L^2(\Omega))$ is called an admissible
control if $y(T; y_0, \hat u)=0$.
\begin{Definition}\label{definition1.2}
Problem $(NP)_{y_0}^{T,p}$ has the bang-bang property if any optimal
control $u^*$ satisfies that $\|\chi_\omega u^*\|_{L^p(0,T;
L^2(\Omega))}=N_p(T,y_0)$ and  $\|\chi_\omega u^*(t)\|\neq 0$ for
a.e. $t\in (0,T)$,  when $1<p<\infty$; while $\|\chi_\omega
u^*(t)\|=N_\infty(T,y_0)$ for a.e. $t\in (0,T)$, when $p=\infty$.
\end{Definition}
\noindent  We treat $N_p(\cdot,y_0)$ as a function of $T$. It is
proved
 that  the limit of $ N_p(T,y_0)$,
 as
 $T$ goes to  $\infty$, exists (see Lemma~\ref{lemma-N-t-lim}). Hence, we can let
\begin{equation}\label{HUANG1.12}
\widehat N_p(y_0)\triangleq\displaystyle\lim_{T\rightarrow \infty}
N_p(T,y_0).
\end{equation}
To ensure the bang-bang property for $(TP)^{M,p}_{y_0}$, we impose
the following condition on  the space $Y_{T,q}$ (with $q$ {\it the
conjugate exponent of $p$, i.e., $\frac{1}{p}+\frac{1}{q}=1$}):
\begin{equation}\label{Y-Z}
 Y_{T,q}=Z_{T,q}\;\;\mbox{for each}\;\; T>0,
\end{equation}
where
\begin{equation}\label{zT1}
Z_{T,q}=\big\{\chi_\omega \psi\in L^q(0,T;L^2(\omega)) \,\mid\, \psi\in C([0,T);
L^2(\Omega)) \mbox{ solves Equation}\;
(\ref{adjoint-equation})\big\}.
\end{equation}

The main results obtained in this paper are as follows.

\begin{Theorem}\label{iso}
 Let $T>0$. Let $p\in(1, \infty]$  and $q$ be the conjugate exponent of $p$.
  $(i)$  When $1< p \leq\infty$, there is a linear isomorphism  $G_p$ from $A_{T,p}$ to
 $Y^*_{T,q}$ (i.e., $G_p$ is linear, one to one and preserves the norms); $(ii)$ When $1< q<\infty$, there is an isomorphism $H_q:\, Y_{T,q}\to A_{T,p}$  defined by
\begin{equation}\label{huangwang1.15}
 H_q(\xi)=\begin{cases}
                y(T; 0, u_{\xi})    , ~& \;\mbox{if}\;\; \xi\neq 0,\\
                     0, ~& \;\mbox{if}\;\; \xi= 0,
                    \end{cases}
\end{equation}
where
\begin{equation}\label{iso-2}
 u_\xi(x,t)=\begin{cases}
                    \|\xi\|^{2-q}_{L^q(0,T;L^2(\omega))}
 \cdot \|\xi(t)\|_\omega^{q-2} \cdot \xi(x,t), ~&(x,t)\in \omega\times(0,T),\\
                     0, ~&(x,t)\in (\Omega\backslash\omega)\times(0,T).
                    \end{cases}
 \end{equation}

\end{Theorem}

\begin{Theorem}\label{theorem-bang-bang}
Let $y_0\in L^2(\Omega)\setminus\{0\}$.
 Suppose that (\ref{Y-Z}) holds. Then $(TP)^{M,p}_{y_0}$ has the bang-bang property if and only if
 $M>\widehat N_p(y_0)$, where $\widehat N_p(y_0)$ is given by (\ref{HUANG1.12}).

\end{Theorem}

\begin{Remark}\label{WANGreamrk1.1}
$(i)$ It is proved that $\|\xi(t)\|_\omega\neq 0$ for each $t\in
[0,T)$,  when $\xi\in Y_{T,q}\setminus \{0\}$ (see
Lemma~\ref{wanglemma4.3}). Hence, $u_\xi$ in (\ref{iso-2}) is
well-defined; $(ii)$  $H_q$ is nonlinear except
 for the case that $q=2$;
$(iii)$ It is proved that (\ref{Y-Z}) holds for the case where
$a(x,t)=a_1(x)+a_2(t)$ in $\Omega\times\mathbb R^+$, with $a_1\in
L^\infty(\Omega),\,a_2\in L^\infty(\mathbb R^+)$ (see Proposition
\ref{lemma-YT-ZT}). Unfortunately, we don't know if it holds when
$a=a(x,t)$ in $\Omega\times\mathbb R^+$; $(iv)$ It is worth
mentioning that when $y_0\in L^2(\Omega)\setminus\{0\}$,
$(TP)^{M,p}_{y_0}$ has optimal controls if and only if $M>\widehat
N_p(y_0)$ (see Proposition~\ref{proposition-TP-M-y0}).
\end{Remark}

The attainable subspaces play important roles in the studies of
control problems governed by Equation (\ref{originalequation}) (see,
for instance, \cite{WX-2} where the connection of attainable
subspaces and the stabilization for some periodic evolution system
are provided). To our surprise, the studies on the attainable
subspaces of internally controlled heat equations are quite limited
from the past publications. In \cite{PW0}, the author provided a way
to characterize the elements of a subspace of $A_{T, 2}$, via a
Riesz basis (see Remarks after Theorem 2 on page 530 in \cite{PW0}).
The method used there is borrowed from \cite{LR} and \cite{L}, where
the elements of a subspace of the controlled wave equation (without
the geometric condition imposed on the control region) are
explicitly expressed via a Riesz basis. In \cite{WX-2} (see also
\cite{WX-3}), the authors presented some  properties of attainable
subspaces for some $T$-periodic evolution systems.  Those properties
gives the connection of the space $\bigcup_{t>0} A_{t,\infty}$ and
the spaces $A_{kT,\infty}$, $k\in \mathbb{N}$. The observations
presented in Theorem~\ref{iso} seem to be new. From these
observations, we can see that the structure of the attainable
subspace $A_{T, p}$ is very complicated, since $Y_{T, q}$ is the
completion of the function space $X_{T,q}$ under the norm of
$L^q(0,T;L^2(\omega))$.

 The bang-bang property is one of the most important
properties of time optimal control problems, from which one can
derive the uniqueness of the optimal control (see \cite{HOF} and
\cite{Wang}) and the equivalence of
 the minimal time and norm controls (see \cite{GL}, \cite{WZ} and \cite{WX-1}).
 The bang-bang property  was
first built up in \cite{HOF2} for $(TP)^{M,\infty}_{y_0}$ where
$\omega=\Omega$ and $a$ is time-invariant. When $p\in (1,\infty)$,
 $\omega\subset\subset\Omega$ and $a$ is time-invariant,
the bang-bang property of  $(TP)^{M,p}_{y_0}$ was studied in
\cite{GL}. It was first realized in  \cite{MS} (partially inspired
by the work \cite{EJPGS}) that the bang-bang property of
$(TP)^{M,\infty}_{y_0}$, where $w\subset\subset\Omega$ and $a$ is
time-invariant, can be derived from the E-controllability: {\it For
each $T>0$, each measurable subset $E\subset(0,T)$ of positive
measure and each $y_0\in L^2(\Omega)$, there is
 a control $u\in L^\infty(0,T;L^2(\Omega))$ with $\|u\|_{L^\infty(0,T;L^2(\Omega))}\leq C(\Omega,\omega,T,E)\|y_0\|$
 s.t. $y(T;y_0,\chi_Eu)=0$ (where $\chi_E$ is the characteristic function of
 $E$)}. In fact,  once the E-controllability holds, one can easily
prove the bang-bang property by  contradiction, through using the
E-controllability and the time-invariance of the system. The
E-controllability was first built up for the case where $a=0$ (see
\cite{Wang}), and then was extended to the case where $a$ is
time-varying (see \cite{PW} and \cite{PWCZ}).
 Here, we would like
to mention that when $\omega\subset\subset\Omega$, the bang-bang
property for some time-invariant semilinear heat equations was first
built up in \cite{PWCZ}, via a very smart way.
  However,  we are not able to use the methods in \cite{GL} and
  \cite{MS} (see also \cite{Wang}) to derive the bang-bang property
  of $(TP)^{M,p}_{y_0}$ where  $\omega\subset\subset\Omega$ and $a$ is time-varying (even for the special case where $\omega\subset\subset\Omega$ and
  $a(x,t)=a_1(x)+a_2(t)$ with $a_1\in L^\infty(\Omega)$ and $a_2\in L^\infty(\mathbb{R}^+)$).
Our Theorem~\ref{theorem-bang-bang} provides the sufficient
(\ref{Y-Z}) to ensure the bang-bang property for the time-varying
case. This theorem,  along with Proposition \ref{lemma-YT-ZT},
implies the bang-bang property for the above-mentioned special case.
   About works on the
time optimal control problems, we would like to mention the papers
\cite{AEWZ, Raymond,  Barbu,  HOF, HOF2, ITO, KL1, KL2, LW, LW-1,
QL, MRT, MS,  PW,  PWCZ, PWZ, EJPGS, C.S.E.T, Wang, WX-0, WZ, Y,
CZ-1, CZ-2, Guo-Zh} and the references therein.

The rest of the paper is organized as follows: Section 2 proves
Theorem \ref{iso}. Section 3 presents some properties on
$(NP)_{y_0}^{T,p}$. Section 4 provides the proof of  Theorem
\ref{theorem-bang-bang}.

\section{Attainable subspaces}

The aim of this section is to prove Theorem~\ref{iso}. We start with
proving its first part.

\begin{proof}[Proof of the  part $(i)$ of Theorem~\ref{iso}.]
 First of all,  from  equations (\ref{originalequation}) and (\ref{adjoint-equation}), one can easily check that
 \begin{equation}\label{jiaxu2.1}
 \int_0^T\langle v(t), \chi_\omega\varphi(t; T, z)\rangle \,\mathrm{d}t =
 \langle y(T; 0, v),z\rangle\;\;\mbox{for all}\;\; z\in L^2(\Omega),
 v\in L^p(0,T; L^2(\Omega)).
 \end{equation}
 Let $y_T\in A_{T,p}$. Then $y_T=y(T;0,\hat{u})$ for some $\hat u\in L^p(0,T;L^2(\Omega))$. Define
 $
  \mathcal F_{y_T, q}: X_{T,q}\longrightarrow \mathbb R
 $
 by setting
 \begin{equation}\label{attain-1}
    \mathcal F_{y_T, q}(\chi_\omega \varphi(\cdot;T,z)) = \int_0^T \langle \hat u(t),\chi_\omega\varphi(t;T,z)
    \rangle \,\mathrm{d}t\;\;\mbox{for each}\;\; z\in L^2(\Omega).
 \end{equation}
From (\ref{jiaxu2.1}) and (\ref{attain-1}), one can easily check
that  $\mathcal F_{y_T, q}$ is well-defined and linear. Meanwhile,
using the H$\ddot{o}$lder's inequality to the right side of
  (\ref{attain-1}), we see that  $\mathcal F_{y_T, q}$ is bounded.
  Thus $\mathcal F_{y_T, q}\in X^*_{T,q}$. Since $X^*_{T,q}=Y^*_{T,q}$ (see (\ref{wang1.8})), we have
\begin{equation}\label{WANGHUANG2.2}
  \mathcal F_{y_T, q}\in Y^*_{T,q}.
\end{equation}
Define $G_{p}:~ A_{T,p}\longrightarrow Y^*_{T,q}$ by setting
 \begin{equation}\label{attain-2}
  G_p (y_T)=\mathcal F_{y_T, q} \;\;\mbox{for each}\;\; y_T\in A_{T,p}.
 \end{equation}
 Clearly, $G_p$ is linear. From (\ref{attain-1}) and (\ref{jiaxu2.1}), one can easily verify that  $G_p$ is injective.

 We now prove that  $G_p$ is surjective. Let $i:~ Y_{T,q}\rightarrow L^q(0,T;L^2(\omega))$ be the embedding map and
 $i^*:~ L^p(0,T;L^2(\omega))\rightarrow Y^*_{T,q}$ be the adjoint operator of $i$. We
 claim that
 \begin{equation}\label{attain-3}
  \mbox{Range} (i^*)=Y^*_{T,q},\;\;\mbox{i.e.,}\;\;i^*\;\;\mbox{is
  surjective}.
 \end{equation}
 By the Hahn-Banach theorem, for each $\mathcal F\in Y^*_{T,q}$,  there is a $\widetilde {\mathcal F}\in \big(L^q(0,T;L^2(\omega))\big)^*$ s.t.
\begin{equation}\label{attain-6}
 \widetilde{\mathcal F}(\xi)=\mathcal F(\xi)\;\;\mbox{for each}\;\;\xi\in Y_{T,q}
\end{equation}
and
\begin{equation*}
 \|\widetilde {\mathcal F} \|_{L\left(L^q(0,T;L^2(\omega));\mathbb R\right)} = \|\mathcal F\|_{Y_{T,q}^*}.
\end{equation*}
According to  the Riesz representation theorem, there is a $\hat
v\in L^p(0,T;L^2(\omega))$ s.t.
\begin{equation}\label{attain-7}
    \widetilde {\mathcal F}(\psi)=\int_0^T \langle \hat v(t),\psi(t) \rangle_{\omega}\, \mathrm{d}t\;\;\mbox{for each}\;\;\psi\in L^q(0,T;L^2(\omega)).
\end{equation}
Because $X_{T,q}\subset Y_{T,q}$, it follows from (\ref{attain-6})
and (\ref{attain-7}) that
\begin{equation*}
 \mathcal F(\chi_\omega\varphi(\cdot;T,z)) = \int_0^T \langle \hat v(t),\chi_\omega\varphi(t;T,z) \rangle_\omega \,\mathrm{d}t\;\;\mbox{for each}\;\;z\in L^2(\Omega).
\end{equation*}
Thus, it holds that
\begin{eqnarray*}
  \langle i^*(\hat v),\chi_\omega\varphi(\cdot,T,z) \rangle_{Y^*_{T,q},Y_{T,q}} &=& \langle \hat v,
   \chi_\omega\varphi(\cdot;T,z) \rangle_{L^p(0,T;L^2(\omega)),L^q(0,T;L^2(\omega))} \\
    &=& \int_0^T \langle \hat v(t),\chi_\omega\varphi(t;T,z) \rangle_\omega \;\mathrm{d}t\\
    &=&\mathcal F(\chi_\omega\varphi(\cdot;T,z))\;\;\mbox{for each}\;\;z\in L^2(\Omega).
\end{eqnarray*}
This, along with  (\ref{wang1.8}), yields that $i^*(\hat v)=\mathcal
F, $ which leads to (\ref{attain-3}).

By  (\ref{attain-3}),  for each $\mathcal F\in Y^*_{T,q}$, we can
find a  $v\in L^p(0,T;L^2(\omega))$ s.t. $i^*(v)=\mathcal F$. We
extend $v$ over $\Omega\times (0, T)$ by setting it to be $0$ on
$(\Omega\setminus\omega)\times (0,T)$, and denote the extension by
$\widetilde v$. Then  $\widetilde v\in L^p(0,T;L^2(\Omega))$ and
$\widetilde y_T\triangleq y(T;0,\widetilde v)\in A_{T,p}$. Moreover,
\begin{eqnarray}\label{attain-4}
  \mathcal F_{\widetilde y_T, q}(\chi_\omega\varphi(\cdot;T,z)) &=& \int_0^T \langle \widetilde v(t), \chi_\omega\varphi(t;T,z) \rangle \;\mathrm{d}t \nonumber\\
      &=& \int_0^T \langle v(t), \chi_\omega\varphi(t;T,z) \rangle_\omega \, \mathrm{d}t\;\;\mbox{for each}\;\;z\in L^2(\Omega)
\end{eqnarray}
and
\begin{equation}\label{attain-5}
 G_p(\widetilde y_T) = \mathcal F_{\widetilde y_T, q}.
\end{equation}
On the other hand, since $i^*(v)=\mathcal F$ and
\begin{eqnarray*}
  \langle i^*(v),\chi_\omega\varphi(\cdot;T,z) \rangle_{Y^*_{T,q},Y_{T,q}} &=& \langle v,i(\chi_\omega\varphi(\cdot;T,z)) \rangle_{L^p(0,T;L^2(\omega)),L^q(0,T;L^2(\omega))} \\
   &=& \langle v,\chi_\omega\varphi(\cdot;T,z) \rangle_{L^p(0,T;L^2(\omega)),L^q(0,T;L^2(\omega))} \\
   &=& \int_0^T \langle v(t),\chi_\omega\varphi(t;T,z) \rangle_\omega\, \mathrm{d}t\;\;\mbox{for each}\;\;z\in L^2(\Omega),
\end{eqnarray*}
 we have
 \begin{equation*}
  \mathcal F(\chi_\omega\varphi(\cdot;T,z)) = \int_0^T \langle v(t),\chi_\omega\varphi(t;T,z) \rangle_{\omega} \,\mathrm{d}t\;\;\mbox{for each}\;\;z\in L^2(\Omega).
 \end{equation*}
 This, along with (\ref{attain-4}), (\ref{wang1.8}) and (\ref{attain-5}), yields
 $  \mathcal F=\mathcal F_{\widetilde y_T, q}$ and $G_p(\widetilde y_T)=\mathcal F$.
  Hence, $G_p$ is surjective.

 Finally, we show that
\begin{equation}\label{attain-8}
 \|y_T\|_{A_{T,p}} = \|\mathcal F_{y_T, q}\|_{Y^*_{T,q}}\;\;(\mbox{i.e.,}\;\;\|y_T\|_{A_{T,p}} = \|G_p(y_T)\|_{Y^*_{T,q}},)\;\;\mbox{for each}\;\; y_T\in
 A_{T,p}.
\end{equation}
Let $y_T\in A_{T,p}$. Arbitrarily take a  $\hat u\in
L^p(0,T;L^2(\Omega))$ such that  $y_T=y(T;0,\hat u)$. From
 (\ref{attain-1}) and (\ref{wang1.8}), it follows
that
\begin{equation}\label{attain-9-1}
 \langle \mathcal F_{y_T, q},\xi \rangle_{Y^*_{T,q},Y_{T,q}} = \int_0^T \langle \hat u(t),\xi(t) \rangle_{w} \,\mathrm{d}t\;\;\mbox{for each}\;\;\xi\in Y_{T,q}.
\end{equation}
Hence, it holds that
\begin{equation*}\label{attain-10}
 \|\mathcal F_{y_T, q} \|_{Y^*_{T,q}} \leq \|\hat u\|_{L^p(0,T;L^2(\omega))} \leq \|\hat u\|_{L^p(0,T;L^2(\Omega))},
\end{equation*}
which, as well as (\ref{jiaxu1.6}),  leads to
\begin{eqnarray}\label{attain-11}
 \|\mathcal F_{y_T, q} \|_{Y^*_{T,q}} &\leq& \inf \{\|u\|_{L^p(0,T;L^2(\Omega))}\,|~y(T;0,u)=y_T\}  =   \|y_T\|_{A_{T,p}}.
\end{eqnarray}
Conversely, we fix a $\hat{v}\in L^p(0,T;L^2(\Omega))$ s.t.
$y(T;0,\hat{v})=y_T$.
 It follows from (\ref{attain-9-1}) that
\begin{equation}\label{attain-12}
 \Big| \int_0^T \langle \hat{v}(t),\xi(t) \rangle_\omega \mathrm{d}t \Big| \leq \|\mathcal F_{y_T, q} \|_{Y^*_{T,q}} \cdot \|\xi\|_{Y_{T,q}}
\;\;\mbox{for each}\;\;\xi\in Y_{T,q}.
\end{equation}
  Define
$G^{\hat{v}}:~Y_{T,q}\rightarrow \mathbb R$ by
\begin{equation}\label{Gu}
 G^{\hat v}(\xi)=\int_0^T \langle \hat v(t),\xi(t) \rangle_\omega\, \mathrm{d}t\;\;\mbox{for each}\;\;\xi\in Y_{T,q}.
\end{equation}
By (\ref{Gu}) and (\ref{attain-12}), $G^{\hat v}\in Y^*_{T,q}$ and $
\| G^{\hat v} \|_{Y^*_{T,q}} \leq \| \mathcal F_{y_T, q}
\|_{Y^*_{T,q}}$. Then, by the Hahn-Banach theorem, the Riesz
representation theorem and (\ref{Gu}), there is a $v\in
L^p(0,T;L^2(\omega))$ s.t.
\begin{equation}\label{attain-13}
 \|v\|_{L^p(0,T;L^2(\omega))} \leq \| \mathcal F_{y_T,q} \|_{Y^*_{T,q}}
\end{equation}
and
\begin{equation}\label{attain-14}
 \int_0^T \langle \hat{v}(t),\xi(t) \rangle_\omega \,\mathrm{d}t = \int_0^T \langle v(t),\xi(t) \rangle_\omega \,\mathrm{d}t\;\;\mbox{for each}\;\;\xi\in Y_{T,q}.
\end{equation}
Since $X_{T,q}\subset Y_{T,q}$, we  have from (\ref{attain-14}) that
\begin{equation}\label{WAng2.19}
\int_0^T \langle \hat{v}(t),\chi_\omega\varphi(t;T,z) \rangle\;
\mathrm{d}t = \int_0^T \langle  \widetilde
v(t),\chi_\omega\varphi(t;T,z) \rangle \;\mathrm{d}t\;\;\mbox{for
each}\;\;z\in L^2(\Omega),
\end{equation}
where $\widetilde v$ is the extension of $v$ over $\Omega\times
(0,T)$ such that $\widetilde v=0$ over $\Omega\setminus\omega \times
(0,T)$.  Since $y(T;0,\hat{v})=y_T$, one can easily check, by using
(\ref{WAng2.19}) and (\ref{jiaxu2.1}), that $y_T=y(T;0,\widetilde
v)$. This, along with  (\ref{jiaxu1.6}) and (\ref{attain-13}), leads
to
\begin{equation}\label{attain-15}
 \|y_T\|_{A_{T,p}} \leq \|\widetilde v\|_{L^p(0,T;L^2(\Omega))}=\| v\|_{L^p(0,T;L^2(\omega))} \leq \| \mathcal F_{y_T} \|_{Y^*_{T,q}}.
\end{equation}
Now, (\ref{attain-8}) follows from (\ref{attain-11}) and
(\ref{attain-15}). This completes the proof of the part $(i)$ of
Theorem~\ref{iso}.

\end{proof}

 To prove the part $(ii)$ of
Theorem~\ref{iso}, we need to present some properties on $Y_{T,q}$.

\begin{Lemma}\label{wanglemma4.3}
Let $1\leq q<\infty$. $(i)$ $Y_{T,q}$ consists of all such functions
$\chi_\omega\varphi\in L^q(0,T;L^2(\omega))$  that  $\varphi\in C([0,T); L^2(\Omega))$  solves Equation (\ref{adjoint-equation}),
  and
$\chi_\omega\varphi=\lim_{n\rightarrow\infty}\chi_\omega\varphi(\cdot;
T, z_n)$ for some sequence  $\{z_n\}\subset L^2(\Omega)$, where the
limit is taken in $L^q(0,T;L^2(\omega))$; $(ii)$ When $\xi\in
Y_{T,q}\setminus\{0\}$, it holds that $\|\xi(t)\|_{\omega}\neq 0$
for each $t\in [0,T)$.

\end{Lemma}

\begin{proof}
$(i)$ Let $\xi\in Y_{T,q}$. By (\ref{wang1.8}), there is a sequence
$\{z_n\} $ in $L^2(\Omega)$   such that
\begin{equation}\label{wang4.7}
\chi_\omega \varphi(\cdot; T,z_n)\rightarrow \xi\;\;\mbox{strongly
in}\;\; L^q(0,T; L^2(\omega)).
\end{equation}
In particular, $\{\chi_\omega\varphi(\cdot; T, z_n)\} $ is bounded in
${L^q(0,T; L^2(\omega))}$. Let $\{T_k\}\subset (0,T)$ such that $T_k
\nearrow T$ (i.e, $T_k$ strictly monotonically converges to $T$ from
the left). Given a $k\in \mathbb{N}$, by the observability estimate
(see, for instance, \cite{FZ}),
\begin{eqnarray}\label{newyear2.21}
\|\varphi(T_{k+1}; T, z_n)\| &\leq& C(k)\|\chi_\omega\varphi(\cdot;
T,z_n)\|_{L^1(T_{k+1},T; L^2(\omega))}\nonumber\\
 &\leq& C(k)\|\chi_\omega\varphi(\cdot;
T,z_n)\|_{L^q(0,T; L^2(\omega))} \leq C(k)\;\;\mbox{for all}\;\;
n\in \mathbb{N},
\end{eqnarray}
where $C(k)$ stands for  a positive constant depending on $k$ but
independent of $n$, which may vary in different contexts.
Arbitrarily take two subsequences $\{\varphi(\cdot; T,
z_{n_{l_1}})\}$ and $\{\varphi(\cdot; T, z_{n_{l_2}})\}$ from
$\{\varphi(\cdot; T, z_n)\} $. By (\ref{newyear2.21}) and the
properties of heat equations, there are two subsequences of
$\{\varphi(\cdot; T, z_{n_{l_1}})\}$ and $\{\varphi(\cdot; T,
z_{n_{l_2}})\}$ respectively, denoted in the same way, such that
$$
\varphi(\cdot; T, z_{n_{l_1}})\rightarrow \hat \varphi_{k,1}(\cdot);
\;\varphi(\cdot; T, z_{n_{l_2}})\rightarrow \hat
\varphi_{k,2}(\cdot)\;\;\mbox{strongly in}\;\; C([0,T_k];
L^2(\Omega)),
$$
where $\hat \varphi_{k,1}$ and  $\hat \varphi_{k,2}$ solve equation
(\ref{adjoint-equation}) (with $T$ being replaced by $T_k$). These,
along with (\ref{wang4.7}), yield that
$$
\chi_\omega\hat \varphi_{k,1}(t)=\chi_\omega\hat
\varphi_{k,2}(t)=\xi(t)\;\;\mbox{ for a.e.}\; t\in[0, T_k].
$$
Then by the unique continuation estimate for heat equations built up
in  \cite{PWCZ} (see also \cite{PW1}), we have
$$
\hat \varphi_k\triangleq \hat \varphi_{k,1}=\hat
\varphi_{k,2}\;\;\mbox{over}\;\; [0, T_k].
$$
Hence, it holds that
\begin{equation}\label{newyear2.22}
\varphi(\cdot; T, z_n)\rightarrow \hat
\varphi_k(\cdot)\;\;\mbox{in}\;\;C([0, T_k]; L^2(\Omega));
\;\;\chi_\omega\hat \varphi_k=\xi\;\;\mbox{over}\;\;
(0, T_k).
\end{equation}
Since $k$ in the above was arbitrarily taken from $\mathbb{N}$, it
follows from (\ref{newyear2.22}) that
\begin{equation}\label{newyear2.23}
\hat \varphi_k=\hat \varphi_{k+l};\;\;\chi_\omega\hat
\varphi_k=\xi \;\;\mbox{over}\;\; [0,T_k]\;\;\mbox{for all}\;\;
k,l\in \mathbb{N}.
\end{equation}
We now define the function $\hat \varphi$ over $\Omega\times [0, T)$
by setting
$$
\hat \varphi
=\hat\varphi_k\;\;\mbox{over}\;\; [0, T_k],\;
k=1,2,\dots.
$$
Then by (\ref{newyear2.23}), $\hat \varphi$ is well defined; $\hat
\varphi\in C([0, T); L^2(\Omega))$ solves Equation
(\ref{adjoint-equation}); $\xi=\chi_\omega\hat\varphi$. Clearly,
$\chi_\omega\hat\varphi$ is the limit of $\chi_\omega\varphi(\cdot;
T,z_n)$ in $L^q(0,T; L^2(\omega))$ (see (\ref{wang4.7})). Thus, we
have proved $(i)$.

 $(ii)$ Let  $\xi\in Y_{T,q}\setminus\{0\}$. By $(i)$, there is a function
$\varphi\in C([0,T); L^2(\Omega))$, with $\chi_\omega\varphi\in L^q(0,T;L^2(\omega))$, solving
Equation (\ref{adjoint-equation}), such that
$\xi=\chi_\omega\varphi$. Since $\xi\neq 0$ in $Y_{T,q}$, it holds
that $\varphi\neq 0$ in $L^q(0,T;L^2(\omega))$. Then by the unique
continuation estimate in \cite{PWCZ} (see also \cite{PW},
\cite{PW1}), it follows that $\|\chi_\omega\varphi(t)\|\neq 0$ for
each $t\in [0,T)$. This completes the proof.
\end{proof}

The proof of the part $(ii)$ of Theorem~\ref{iso} needs help from
the following norm optimal control problem $ (NP)_{y_T,p}$:
\begin{equation}\label{JIA2.21}
\inf \big\{ \|u\|_{L^p(0,T;L^2(\Omega))} ~\big|
~y(T;0,u)=y_T\big\},
\end{equation}
where  $p\in (1,\infty]$ and $y_T\in A_{T,p}$. The optimal control
and the admissible control to this problem can be defined by a very
similar way as those for $(NP)_{y_0}^{T,p}$ (see Section 1).
 This problem
is related to the variational problem $(JP)_{y_T,q}$:
\begin{equation}\label{attain-17}
 \inf_{\xi\in Y_{T,q}} J_{y_T,q}(\xi)\triangleq  \inf_{\xi\in Y_{T,q}} \Big(\frac{1}{2} \|\xi\|^2_{L^q(0,T;L^2(\omega))} - \mathcal F_{y_T,q}(\xi)\Big),
 \; \xi\in Y_{T,q},
\end{equation}
where $q$ is the conjugate exponent of $p$ and $\mathcal F_{y_T,q}$
is given by (\ref{attain-1}) (see also (\ref{WANGHUANG2.2})).

\begin{Lemma}\label{norm-Lp}
Let $p\in(1,\infty)$ and $q$ be the conjugate exponent of $p$. $(i)$
When $y_T\in A_{T,p} \setminus\{0\}$, it holds that zero (the origin
of $Y_{T,q}$) is not a minimizer of $ J_{y_T,q}$;  $ J_{y_T,q}$ has
a unique minimizer $\chi_\omega \widehat \varphi$ in $Y_{T,q}$,
  where  $\widehat \varphi\in C([0,T); L^2(\Omega))\cap L^q(0,T; L^2(\omega))$ solves Equation (\ref{adjoint-equation});
    $(NP)_{y_T,p}$ has
  a unique optimal control $\widehat u_{y_T,p}$ given by
 \begin{equation}\label{norm-Lp-u}
  \widehat u_{y_T,p}(t)= \|\chi_\omega\widehat\varphi\|^{2-q}_{L^q(0,T;L^2(\omega))}
  \cdot \|\chi_\omega\widehat\varphi(t)\|^{q-2} \cdot \chi_\omega\widehat\varphi(t),~t\in (0,T);
 \end{equation}
$(ii)$ If $y_T=0$ in $A_{T,q}$, then zero is the unique minimizer of
$J_{0,q}$ and the unique optimal control to $(NP)_{0,p}$ is the null
control.
\end{Lemma}

\begin{proof}
$(i)$  Write $y_T=y(T; 0,u_T)$ for some $u_T\in L^p(0,T;
L^2(\Omega))$. By contradiction, we suppose that zero was a
minimizer. Since $X_{T,q}\subset Y_{T,q}$ (see (\ref{wang1.8})), we
would have
$$
0\leq \frac{J_{y_T,q}(\varepsilon \varphi(\cdot; T,
z))}{\varepsilon}\;\;\mbox{for all}\;\;
\varepsilon>0\;\;\mbox{and}\;\; z\in L^2(\Omega).
$$
This, along with (\ref{attain-17}), (\ref{attain-1}) and
(\ref{jiaxu2.1}), yields that $<y_T,z>=0$ for all $z\in
L^2(\Omega)$, which contradicts to the fact that $y_T\neq 0$.

 Since $1<q<\infty$, $L^q(0,T;L^2(\omega))$ is  reflexible.
Thus, $Y_{T,q}$, as a  closed subspace of $L^q(0,T;L^2(\omega))$, is
also reflexible. Meanwhile, one can directly check that
$J_{y_T,q}(\cdot)$ is strictly convex and coercive in $Y_{T,q}$.
Hence,  $ J_{y_T,q}$ has a unique minimizer. Furthermore, it follows
from    Lemma~\ref{wanglemma4.3} that this minimizer can be
expressed by  $\chi_\omega\widehat\varphi\in  L^q(0,T; L^2(\omega))$, where $\widehat
\varphi\in C([0,T); L^2(\Omega))$ solves
Equation (\ref{adjoint-equation})
  and verifies $\chi_\omega\widehat\varphi(t)\neq 0$ for all $t\in [0,T)$.

 Since $\mathcal F_{y_T,q}\in Y_{T,q}^*$,  one can easily derive from (\ref{attain-17}) the following
  Euler-Lagrange equation associated with the minimizer $\chi_\omega\widehat\varphi$:
    \begin{equation}\label{Euler-Lp}
  \int_0^T \langle \widehat u_{y_T,p}(t), \xi(t) \rangle_\omega\; \mathrm{d}t-\mathcal F_{y_T,q}(\xi)=0\;\;\mbox{for each}\;\;\xi\in Y_{T,q},
 \end{equation}
where  $\widehat u_{y_T,p}$ is defined by (\ref{norm-Lp-u}).
  From (\ref{Euler-Lp}) and  (\ref{attain-9-1}), it follows  that
 \begin{equation}\label{Euler-Lp-1}
  \int_0^T \langle \widehat u_{y_T,p}(t), \xi(t) \rangle_\omega \;\mathrm{d}t-\int_0^T \langle v(t), \xi(t) \rangle_\omega \;\mathrm{d}t=0
  \;\;\mbox{for each}\;\;\xi\in Y_{T,q},
 \end{equation}
when $v$ is an admissible control  to $(NP)_{y_T,p}$. This, as well
as (\ref{jiaxu2.1}),  in particular, implies
$$
\langle y(T; 0,   \widehat u_{y_T,p}), z\rangle =\langle y_T,
z\rangle \;\;\mbox{for all}\;\; z\in L^2(\Omega),
$$
which leads to
\begin{equation}\label{henhen2.27}
y(T; 0,   \widehat u_{y_T,p})=y_T.
\end{equation}
On the other hand, it follows from (\ref{norm-Lp-u}) that
\begin{equation}\label{u-psi}
 \| \widehat u_{y_T,p}\|_{L^p(0,T;L^2(\Omega))} = \|\chi_\omega\widehat\varphi
 \|_{L^q(0,T;L^2(\omega))}.
\end{equation}
By (\ref{norm-Lp-u}), (\ref{Euler-Lp-1}), with $\xi=\chi_\omega\widehat\varphi$, and
(\ref{u-psi}), for each admissible control $v$ to $(NP)_{y_T,p}$, we see
 \begin{eqnarray}\label{ad-op}
  \Big(\|\widehat u_{y_T,p}\|_{L^p(0,T;L^2(\Omega))}\Big)^2 &=&\int_0^T \langle \widehat u_{y_T,p}(t), \chi_\omega\widehat\varphi(t) \rangle \;\mathrm{d}t
  =\int_0^T \langle v(t), \chi_\omega\widehat\varphi(t) \rangle \;\mathrm{d}t\nonumber\\
   &\leq& \|v\|_{L^p(0,T;L^2(\Omega))} \cdot \|\chi_\omega\widehat\varphi\| _{L^q(0,T;L^2(\omega))}\nonumber\\
  &=& \|v\|_{L^p(0,T;L^2(\Omega))} \cdot \|\widehat u_{y_T,p}\|_{L^p(0,T;L^2(\Omega))}.
 \end{eqnarray}
 Hence, $ \|\widehat u_{y_T,p}\|_{L^p(0,T;L^2(\Omega))} \leq
 \|v\|_{L^p(0,T;L^2(\Omega))}$, when $v$ is an admissible control to  $(NP)_{y_T,p}$.
From this and  (\ref{henhen2.27}),
 $\widehat u_{y_T,p}$ is an optimal control to $(NP)_{y_T,p}$. The
 uniqueness of the optimal control to $(NP)_{y_T,p}$ follows from
 the uniform convexity of $L^p(0,T;L^2(\Omega))$ (with $1<p<\infty$) immediately.

$(ii)$ Its proof is  trivial.   This completes the proof.

\end{proof}

\begin{Lemma}\label{wGanglemma2.3}
Let $\xi\in Y_{T,q}\setminus\{0\}$ with $q\in (1, \infty)$. Then $(i)$
$u_\xi$ (given by (\ref{iso-2})) is the optimal control to
$(NP)_{y_{T,\xi}, q}$ where $y_{T,\xi}\triangleq y(T; 0, u_\xi)$;
$(ii)$ $\xi$ is the minimizer of $J_{y_{T,\xi}, q}$.
\end{Lemma}
\begin{proof}
$(i)$ Given an admissible control  $v$ to $(NP)_{y_{T,\xi}, q}$, it
follows from (\ref{attain-9-1}) that
\begin{equation}\label{wanghuang2.29}
\mathcal{F}_{y_{T,\xi},q}(\eta)=\int_0^T\langle u_\xi(t), \eta(t)
\rangle_\omega \;\mathrm{d}t= \int_0^T\langle v(t), \eta(t) \rangle_\omega
\;\mathrm{d}t\;\;\mbox{for each}\;\; \eta\in Y_{T,q}.
\end{equation}
From (\ref{iso-2}), we have
 \begin{equation}\label{jiaxu2031}
 \|u_\xi\|_{L^p(0,T;
L^2(\Omega))}=\|\xi\|_{L^q(0,T;L^2(\omega))}.
\end{equation}
Taking $\eta=\xi$ in the second equality of (\ref{wanghuang2.29}),
using (\ref{iso-2}), (\ref{jiaxu2031}) and the H$\ddot{o}$lder
inequality, we get $\|u_\xi\|_{L^p(0,T; L^2(\Omega))}\leq
\|v\|_{L^p(0,T; L^2(\Omega))}$. Hence, $u_\xi$ is the optimal
control to $(NP)_{y_{T,\xi}, q}$.

$(ii)$ By (\ref{attain-17}), the first equality  of
(\ref{wanghuang2.29}) and (\ref{iso-2}), after some simple
computations involving the Cauchy-Schwartz and the  H$\ddot{o}$lder
inequalities, one can get that $J_{y_{T,\xi},q}(\xi)\leq
J_{y_{T,\xi},q}(\eta)$ for all $\eta\in Y_{T,q}$, i.e, $\xi$ is the
minimizer of $J_{y_{T,\xi},q}$. This completes the proof.

\end{proof}

\begin{Remark}
 Unfortunately, we don't know how to get the similar results in
Lemma~\ref{norm-Lp} and Lemma~\ref{wGanglemma2.3} for the case where
$p=\infty$.
\end{Remark}

Now we continue the proof of  Theorem~\ref{iso}.

\begin{proof} [Proof of the  part $(ii)$ of Theorem~\ref{iso}.]

 Let
$H_q$ be defined by (\ref{huangwang1.15}). We first show that $H_q$
is injective. Let $\xi\neq \eta$ in $Y_{T,q}$. In the case that both
$\xi$ and $\eta$  are not zero, we suppose by contradiction that
$H_q(\xi)=H_q(\eta)$. Then, $y_{T,\xi}\triangleq y(T; 0, u_\xi)=y(T;
0, u_\eta)\triangleq y_{T,\eta}$. By Lemma~\ref{wGanglemma2.3}, both
$\xi$ and $\eta$ are the unique minimizer of $J_{y_{T,\xi},q}$. Thus
$\xi=\eta$ which leads to a contradiction. Hence, $H_p(\xi)\neq
H_p(\eta)$ when $\xi\neq\eta$ in $Y_{T,q}\setminus \{0\}$. In the
case where $\xi\neq 0$ and $\eta=0$, it suffices to show that
$H_q(\xi)\neq 0$. By contradiction, we suppose that $0=H_q(\xi)$. By
(\ref{huangwang1.15}), we have  $y(T;0,u_\xi)= 0$, where $u_\xi$ is
given by (\ref{iso-2}). According to Lemma~\ref{wGanglemma2.3},
$u_\xi$ is the optimal control to $(NP)_{0,p}$. This, along with
$(ii)$ of Lemma~\ref{norm-Lp}, yields that $u_\xi=0$ in
$L^p(0,T;L^2(\Omega))$. However, it follows from
Lemma~\ref{wanglemma4.3}, as well as (\ref{iso-2}), that
$\|u_\xi(t)\|_\omega\neq 0$ when $t\in [0, T)$. This leads to a
contradiction. In summary, we conclude that $H_q$ is injective.

We next show that $H_q$ is surjective. Given $y_T\in A_{T,
p}\setminus\{0\}$, let $\xi$ be the minimizer of $J_{y_T,q}$ in
$Y_{T,q}$. By Lemma~\ref{norm-Lp}, $u_\xi$ (given by (\ref{iso-2}))
is the optimal control to $(NP)_{y_T,p}$. Hence, $ H_q(\xi)= y(T; 0,
u_\xi)=y_T$. This, along with the fact that $H_q(0)=0$, indicates
that $ H_q$ is surjective.

Finally, we show that $H_q$ preserves the norms. Given $\xi\in
Y_{T,q}\setminus\{0\}$, it holds that
$H_q(\xi)=y(T;0,u_\xi)\triangleq y_{T,\xi}$. Since $u_\xi$ is the
optimal control to $(NP)_{y_{T,\xi}, p}$ (see
Lemma~\ref{wGanglemma2.3}), we derive  from (\ref{jiaxu1.6}) that
\begin{equation*}
\|H_q(\xi)\|_{A_{T,p}}=\| y_{T,\xi}\|_{A_{T,p}} =\|u_\xi\|_{L^p(0,T;
L^2(\Omega))}.
\end{equation*}
which, together with (\ref{jiaxu2031}) and (\ref{WANG1.6}), leads to
$\|H_q(\xi)\|_{A_{T,p}}=\|\xi\|_{Y_{T,q}}$. This completes the proof
of the part $(ii)$ of Theorem~\ref{iso}.

\end{proof}

\section{Some properties on $N_p(T, y_0)$}

This section presents some properties on $N_p(T ,y_0)$ (given by
(\ref{NP})).  These properties will be used in the proof of
Theorem~\ref{theorem-bang-bang}. We  focus  on the case where
$y_0\neq 0$, since $N_p(\cdot, 0)\equiv 0$.
\begin{Lemma}\label{NP-YT}
  Let   $p\in(1, \infty]$ and $q$ be the conjugate exponent of $p$. Then
\begin{equation}\label{N(T,y0)}
N_p(T,y_0)=\displaystyle\sup_{z\in L^2(\Omega)\backslash \{0\}}
\dfrac{\langle y(T; y_0,0),z
\rangle}{\|\chi_\omega\varphi(\cdot;T,z)\|_{L^q(0,T;L^2(\Omega))}}\;\;\mbox{for
all}\;\; T>0,\; y_0\in L^2(\Omega)\setminus\{0\}.
\end{equation}
\end{Lemma}

\begin{proof}
Let $y_0\in L^2(\Omega)\setminus\{0\}$ and $T>0$. Write
$y_T\triangleq -y(T; y_0,0)$. From the $L^\infty$-null
controllability (see \cite{FZ} or \cite{PWCZ}), it follows that
$y_T\in A_{T,p}$.  Clearly, $y(T;y_0,u)=0$ if and only if
$y(T;0,u)=y_T$. These, along with  (\ref{NP}) and (\ref{jiaxu1.6}),
yields that
 \begin{eqnarray}\label{wang3.2}
 N_p(T,y_0) &=& \inf\big\{\|u\|_{L^p(0,T; L^2(\Omega))}\;|\; y(T; y_0, u)=0  \big\} \nonumber\\
&=& \inf\big\{\|u\|_{L^p(0,T; L^2(\Omega))}\;|\; y(T; 0, u)=y_T  \big\} \nonumber\\
 &=& \|y_T\|_{A_{T,p}}.
\end{eqnarray}
Let $\hat u\in L^p(0,T; L^2(\Omega))$ be such that $y(T; 0,\hat
u)=y_T$. By (\ref{attain-1}) and (\ref{jiaxu2.1}), it follows that
$$
\mathcal{F}_{y_T,q}(\chi_\omega\varphi(\cdot;T,z))=\langle y_T,
z\rangle\;\;\mbox{for all}\;\; z\in L^2(\Omega).
$$
This, combined with (\ref{WANG1.6}) and (\ref{wang1.8}), yields that
\begin{equation}\label{wang3.3}
\|\mathcal{F}_{y_T,q}\|_{Y^*_{T,q}}=\displaystyle\sup_{z\in
L^2(\Omega)\backslash \{0\}} \dfrac{\langle y(T; y_0,0),z
\rangle}{\|\chi_\omega\varphi(\cdot;T,z)\|_{L^q(0,T;L^2(\Omega))}}.
\end{equation}
By (\ref{wang3.2}), (\ref{attain-8})  and (\ref{wang3.3}), we are
led to (\ref{N(T,y0)}). This completes the proof.

\end{proof}

 The studies on $N_p(T, y_0)$ are closely related to  the variational problem $(JP)_{y_0}^{T,q}$:
\begin{equation}\label{J-y0}
 V_q(T,y_0)\triangleq \displaystyle\inf_{\chi_\omega\varphi\in
 Y_{T,q}}J_{y_0}^{T,q}(\chi_\omega\varphi)\triangleq\frac{1}{2} \Big(
\|\chi_\omega\varphi\|_{L^q(0,T;L^2(\Omega))}\Big)^2
 + \langle y_0,\varphi(0) \rangle.
\end{equation}
By $(i)$ of Lemma~\ref{wanglemma4.3}, $J_{y_0}^{T,q}$ is
well-defined over $ Y_{T,q}$.

\begin{Lemma}\label{proposition-V-T-y0}
 Let   $p\in(1,\infty]$ and $q$ be the conjugate exponent of $p$. Then
 \begin{equation}\label{V-T-y0-N-T-y0}
V_q(T,y_0)=-\frac{1}{2}N_p(T,y_0)^2\;\;\mbox{for all}\;\;
T>0\;\;\mbox{and}\;\; y_0\in L^2(\Omega)\setminus\{0\}.
\end{equation}
\end{Lemma}

\begin{proof}
We first prove that
\begin{equation}\label{proof-V-T-y0-N-T-y0-geq}
V_q(T,y_0)\geq-\frac{1}{2}N_p(T,y_0)^2\;\;\mbox{for all}\;\;
T>0\;\;\mbox{and}\;\; y_0\in L^2(\Omega)\setminus\{0\}.
\end{equation}
From the unique continuation estimate of heat equations (see, for
instance, \cite{PWCZ}, \cite{PW}), it follows that
$\chi_\omega\varphi(t;T,z)\neq 0$, when  $z\in
L^2(\Omega)\backslash\{0\}$ and  $t\in [0,T)$. This, along with
(\ref{J-y0}), indicates  that
\begin{eqnarray}\label{proof-V-T-y0-N-T-y0-geq-1}
 & &J^{T,q}_{y_0}(\chi_\omega\varphi(\cdot;T,z))\nonumber\\
&=&\frac{1}{2}\Big[\|\chi_\omega
\varphi(\cdot;T,z)\|_{L^q(0,T;L^2(\Omega))}+\frac{\langle
y_0,\varphi(0;T,z)\rangle} {\|\chi_\omega
\varphi(\cdot;T,z)\|_{L^q(0,T;L^2(\Omega))}}\Big]^2\nonumber\\
&&-\frac{1}{2} \Big[\frac{\langle
y_0,\varphi(0;T,z)\rangle}{\|\chi_\omega
\varphi(\cdot;T,z)\|_{L^q(0,T;L^2(\Omega))}}\Big]^2\nonumber\\
&\geq& -\frac{1}{2} \Big[\frac{\langle y(T; y_0,0),
z\rangle}{\|\chi_\omega
\varphi(\cdot;T,z)\|_{L^q(0,T;L^2(\Omega))}}\Big]^2~\mbox{ for each }\,z\in
L^2(\Omega)\setminus\{0\}.
\end{eqnarray}
Meanwhile, it follows from Lemma \ref{NP-YT} that
\begin{equation}\label{proof-V-T-y0-N-T-y0-geq-2}
N_p(T,y_0)\geq \frac{\langle y(T; y_0,0), z\rangle}{\|\chi_\omega
\varphi(\cdot; T,z)\|_{L^q(0,T;L^2(\Omega))}} ~\textrm{ for each
}z\in L^2(\Omega)\backslash\{0\}
\end{equation}
and
\begin{eqnarray}\label{proof-V-T-y0-N-T-y0-geq-3}
-N_p(T,y_0)&=&\inf_{z\in
L^2(\Omega)\backslash\{0\}}\Big[\frac{\langle y(T; y_0,0),
-z\rangle}{\|\chi_\omega \varphi(\cdot; T, -z)
\|_{L^q(0,T;L^2(\Omega))}}\Big]\nonumber\\
&=&\inf_{z\in L^2(\Omega)\backslash\{0\}}\Big[\frac{\langle y(T;
y_0,0), z\rangle}{\|\chi_\omega \varphi(\cdot;
T,z)\|_{L^q(0,T;L^2(\Omega))}}\Big].
\end{eqnarray}
By (\ref{proof-V-T-y0-N-T-y0-geq-3}), we find that
\begin{equation}\label{proof-V-T-y0-N-T-y0-geq-4}
-N_p(T,y_0)\leq \frac{\langle y(T; y_0,0), z\rangle}{\|\chi_\omega
\varphi(\cdot; T,z)\|_{L^q(0,T;L^2(\Omega))}}\;  \textrm{ for all
}z\in L^2(\Omega)\backslash\{0\}.
\end{equation}
From (\ref{proof-V-T-y0-N-T-y0-geq-2}) and (\ref{proof-V-T-y0-N-T-y0-geq-4}), it follows that
\begin{equation*}
\Big|\frac{\langle y(T; y_0,0), z\rangle}{\|\chi_\omega
\varphi(\cdot; T,z)\|_{L^q(0,T;L^2(\Omega))}} \Big|\leq
N_p(T,y_0)\,\textrm{ for all }z\in L^2(\Omega)\backslash\{0\}.
\end{equation*}
Hence,
\begin{equation}\label{proof-V-T-y0-N-T-y0-geq-5}
\sup_{z\in L^2(\Omega)\backslash\{0\}}\Big\{\Big[\frac{\langle y(T;
y_0,0), z\rangle}{\|\chi_\omega \varphi(\cdot;
T,z)\|_{L^q(0,T;L^2(\Omega))}} \Big]^2\Big\}\leq N_p(T,y_0)^2.
\end{equation}

\noindent From (\ref{proof-V-T-y0-N-T-y0-geq-1}) and
(\ref{proof-V-T-y0-N-T-y0-geq-5}), one can easily check that
\begin{eqnarray}\label{wangyuan3.12}
\inf_{z\in
L^2(\Omega)\backslash\{0\}}J^{T,q}_{y_0}(\chi_\omega\varphi(\cdot;T,z))\geq
-\frac{1}{2}N_p(T,y_0)^2.
\end{eqnarray}
 By the same method used to prove the part $(i)$ of
Lemma~\ref{norm-Lp}, we can easily check that $0$ is not the
minimizer of $J_{y_0}^{T,q}$. This, along with
(\ref{J-y0}) and (\ref{wang1.8}), yields that
\begin{equation}\label{V-T-y0}
 V_q(T,y_0)=                 \displaystyle\inf_{z\in
L^2(\Omega)\setminus\{0\}}J^{T,q}_{y_0}(\chi_\omega\varphi(\cdot;T,z))
\;\;\mbox{for all}\;\; y_0\in
L^2(\Omega)\setminus\{0\}\;\;\mbox{and}\;\; T>0.
\end{equation}
From (\ref{wangyuan3.12}) and (\ref{V-T-y0}), we are led to
(\ref{proof-V-T-y0-N-T-y0-geq}).

We next show that
\begin{equation}\label{proof-V-T-y0-N-T-y0-leq}
V_q(T,y_0)\leq-\frac{1}{2}N_p(T,y_0)^2.
\end{equation}

\noindent Clearly, $N_p(T,y_0)>0$ since $y_0\neq 0$. By (\ref{proof-V-T-y0-N-T-y0-geq-3}), given $\varepsilon\in (0,N_p(T,y_0))$, there is a $z_\varepsilon
\in L^2(\Omega)\backslash\{0\}$ such that
\begin{equation}\label{proof-V-T-y0-N-T-y0-leq-1}
\frac{\langle y_0,\varphi(0;T,z_\varepsilon)\rangle}{\|\chi_\omega \varphi(\cdot;T,z_\varepsilon)\|_{L^q(0,T;L^2(\Omega))}}\leq -N_p(T,y_0)+\varepsilon.
\end{equation}
Then,  it follows from (\ref{J-y0}) and
(\ref{proof-V-T-y0-N-T-y0-leq-1}) that for each $\lambda\geq 0$
\begin{eqnarray*}
& &J^{T,q}_{y_0}(\chi_\omega\varphi(\cdot;T,\lambda z_\varepsilon)) \\
&\leq&\frac{1}{2}\left[\lambda\|\chi_\omega\varphi(\cdot;T,z_\varepsilon)\|_{L^q(0,T;L^2(\Omega))}-(N_p(T,y_0)
-\varepsilon)\right]^2-\frac{1}{2}(N_p(T,y_0)-\varepsilon)^2.
\end{eqnarray*}
By taking the infimum  for $\lambda\in \mathbb{R}^+$ on the both
sides of the above inequality, we find that
$$
\inf_{\lambda\in \mathbb
R^+}J^{T,q}_{y_0}(\chi_\omega\varphi(\cdot;T,\lambda z_\varepsilon))
\leq -\frac{1}{2}(N_p(T,y_0) -\varepsilon)^2\;\;\mbox{for each}\;\;
\varepsilon\in(0,N_p(T,y_0)),
$$
which, together with (\ref{V-T-y0}), yields that
$$
V_q(T,y_0)\leq -\frac{1}{2}(N_p(T,y_0) -\varepsilon)^2\;\;\mbox{for
each}\;\; \varepsilon\in(0,N_p(T,y_0)).
$$
Sending $\varepsilon\to 0$ in the above inequality leads to
(\ref{proof-V-T-y0-N-T-y0-leq}).

Finally, (\ref{V-T-y0-N-T-y0})  follows from
(\ref{proof-V-T-y0-N-T-y0-geq}) and (\ref{proof-V-T-y0-N-T-y0-leq}).
This completes the proof.
\end{proof}

\begin{Lemma}\label{JIAlemma3.3}
Let $p\in (1, \infty]$ and $q$ be the conjugate exponent of $p$. Let
$T>0$ and $y_0\in L^2(\Omega)\setminus\{0\}$.  Write $y_T\triangleq -
y(T; y_0,0)$. Then $(i)$ $J_{y_T,q}(\cdot) =J_{y_0}^{T,q}(\cdot)$
over $Y_{T,q}$; $(ii)$ Problems $(NP)_{y_T,p}$  and
$(NP)_{y_0}^{T,p}$ have the same optimal controls. (Here,
$J_{y_T,q}$ and $(NP)_{y_T,p}$ are defined by (\ref{attain-17}) and
(\ref{JIA2.21}) respectively.)
\end{Lemma}
\begin{proof}
$(i)$ By the $L^\infty$-null controllability (see, for instance,
\cite {FZ} or \cite{PWCZ}), one can show that $y_T\in A_{T,q}$.
Thus, $y_T=y(T;0,\widehat u)$ for some $\widehat u\in
L^p(0,T;L^2(\Omega))$. This, together with  (\ref{attain-1}) and
(\ref{jiaxu2.1}), indicates that for each $z\in L^2(\Omega)$,
\begin{eqnarray}\label{J-q-1}
 \mathcal F_{y_T,q} (\chi_\omega\varphi(\cdot;T,z))= \int_0^T\langle
 \hat u(t), \chi_\omega\varphi(t; T,z)\rangle \;\mathrm{d}t
=\langle y_T,z\rangle= -\langle y_0,\varphi(0;T,z) \rangle.
\end{eqnarray}
From this, as well as  the definitions of $J_{y_T,q}$ and
  $J_{y_0}^{T,q}$ (see
(\ref{attain-17}) and (\ref{J-y0}) respectively),
Lemma~\ref{wanglemma4.3}, (\ref{WANGHUANG2.2}) and (\ref{wang1.8}), one can easily
get that $J_{y_T,q}(\cdot) =J_{y_0}^{T,q}(\cdot)$ over $Y_{T,q}$.

$(ii)$ The proof is trivial. This completes the proof.

\end{proof}

\begin{Lemma}\label{lemma-NP-control}
 Let $p\in (1, \infty]$. Let  $T>0$ and $y_0\in L^2(\Omega)\setminus\{0\}$.
Then    $(i)$  $(NP)_{y_0}^{T,p}$ holds the bang-bang property;
$(ii)$
  $(NP)_{y_0}^{T,p}$ has a unique optimal
control.
\end{Lemma}

\begin{proof}
 When  $p=\infty$, the results in $(i)$ and $(ii)$  have been proved in \cite[Theorem
3.1]{PW}.

Suppose that   $p\in (1, \infty)$.  Since $y_0\neq 0$, it follows
from the backward uniqueness and the $L^\infty$-null controllability
of heat equations that  $0\neq y_T\triangleq -y(T;y_0,0)\in A_{T,p}$.
Then $(i)$ and $(ii)$ follow from Lemma~\ref{JIAlemma3.3},
Definition~\ref{definition1.2}, $(i)$ of Lemma~\ref{norm-Lp} and
$(ii)$ of Lemma~\ref{wanglemma4.3}.
 This completes
the proof.
\end{proof}

\begin{Lemma}\label{wanglemma4.2}
Let  $q\in (1, \infty)$. Let $T>0$ and $y_0\in
L^2(\Omega)\backslash\{0\}$. Then $(i)$ $0$ is not a minimizer of
$J^{T,q}_{y_0}$; $(ii)$ $J^{T,q}_{y_0}$ has a unique minimizer
$\chi_\omega\widehat\varphi$ in $Y_{T,q}$;  $(iii)$ it holds that
 \begin{equation}\label{JP-value-q}
  V_q(T,y_0)=-\frac{1}{2}  \|\chi_\omega\widehat\varphi\|_{L^q(0,T;L^2(\Omega))}^2,
 \end{equation}
 where $V_q(T,y_0)$ is given by (\ref{J-y0}).
\end{Lemma}

\begin{proof}
Let $p$ be the conjugate exponent of $q$. Write $y_T\triangleq-y(T;
y_0,0)$. Clearly,  $0\neq y_T\in A_{T,p}$. Then, according to
Lemma~\ref{JIAlemma3.3} and $(i)$ of Lemma~\ref{norm-Lp},
$J^{T,q}_{y_0}$ has a unique minimizer $\chi_\omega\widehat\varphi$
in $Y_{T,q}$. By Lemma~\ref{wanglemma4.3} and  (\ref{J-y0}), one
can easily check the following  Euler-Lagrange equation associated
with $\chi_\omega\widehat\varphi$:

\begin{equation}\label{J-q-Euler-equation}
\langle y_0,\varphi(0)\rangle+\int_0^T\langle \hat u(t),\chi_\omega
\varphi(t)\rangle\,\mathrm{d}t=0\;\;\mbox{for
all}\;\;\chi_\omega\varphi\in Y_{T,q},
\end{equation}
where
\begin{equation*}
\widehat u(t) \triangleq
\|\chi_\omega\widehat\varphi\|^{2-q}_{L^q(0,T;L^2(\Omega))} \cdot
\|\chi_\omega\widehat\varphi(t)\|_{\omega}^{q-2} \cdot
\chi_\omega\widehat\varphi(t),~t\in (0,T).
\end{equation*}
Taking $\varphi=\widehat\varphi$ in (\ref{J-q-Euler-equation}) gives
\begin{equation*}
 \langle y_0,\widehat\varphi(0)\rangle+\Big(\|\chi_\omega \widehat\varphi\|_{L^q(0,T;L^2(\Omega))} \Big)^2=0.
\end{equation*}
This, along with  (\ref{J-y0}), leads to (\ref{JP-value-q})  and
completes the proof.

\end{proof}

\begin{Remark}  Since $L^1(0,T; L^2(\omega))$ is not reflexive and its norm is not strictly
convex, the studies on  the functional $J^{T,1}_{y_0}$ is much more
complicated. In the rest of this section, we will show that the
functional $J^{T,1}_{y_0}$ is strictly convex in $Y_{T,1}$. This is
not obvious (see the last paragraph on Page 2940 in \cite{WZ}).
Unfortunately, we do not know if $J^{T,1}_{y_0}$ has a minimizer, in
general. (At least, we do not know how to prove it.) We will show
 the existence of the minimizer for this functional under the assumption  (\ref{Y-Z}).
\end{Remark}

\begin{Lemma}\label{lemma-J-T-y0-unique}
Let $T>0$ and $y_0\in L^2(\Omega)\backslash\{0\}$. Then $(i)$ The
functional $J^{T,1}_{y_0}$ is strictly convex in $Y_{T,1}$.
 Consequently, the minimizer of $J_{y_0}^{T,1}$, if exists, is
 unique; $(ii)$ Zero is not the minimizer of
 $J^{T,1}_{y_0}$.
\end{Lemma}

\begin{proof}
$(i)$   By contradiction, suppose that $J^{T,1}_{y_0}$ was not
strictly convex in $Y_{T,1}$. Then there would be two distinct
$\chi_\omega\varphi_1$ and $\chi_\omega\varphi_2$ in $Y_{T,1}$ and a
$\lambda\in(0,1)$ such that
\begin{equation}\label{proof-J-T-y0-unique-con}
\Big(\int_0^T\|\chi_\omega\varphi_{\lambda}\|\,\mathrm{d}t\Big)^2=
(1-\lambda)\Big(\int_0^T\|\chi_\omega\varphi_1\|\,\mathrm{d}t\Big)^2
+\lambda\Big(\int_0^T\|\chi_\omega\varphi_2\|\,\mathrm{d}t\Big)^2,
\end{equation}
where $\varphi_{\lambda}\triangleq(1-\lambda)\varphi_1+\lambda
\varphi_2$.  We first prove that
\begin{equation}\label{proof-J-T-y0-unique-not-0}
\|\chi_\omega\varphi_1(t)\|\neq 0, \,\|\chi_\omega\varphi_2(t)\|\neq
0 \textrm{ for each }t\in[0,T).
\end{equation}
In fact, if it was not true, then we could suppose, without loss of
generality, that $\chi_\omega\varphi_1(t_0)=0$ for some $t_0\in
[0,T)$. Since both $\varphi_1$ and $\varphi_2$ solve equation
(\ref{adjoint-equation}) (see Lemma~\ref{wanglemma4.3}), it follows
by the unique continuation estimate of heat equations (see, for
instance, \cite{PWCZ}) that $\varphi_1\equiv 0$ over $[0,T]$.
Consequently,
  $\varphi_{\lambda}=\lambda \varphi_2$, which, as well as (\ref{proof-J-T-y0-unique-con}), yields
$$
\lambda\Big(\int_0^T\|\chi_\omega\varphi_2\|\,\mathrm{d}t\Big)^2
=\Big(\int_0^T\|\chi_\omega\varphi_2\|\,\mathrm{d}t\Big)^2.
$$
Because $\lambda\in (0,1)$, the above equality implies that
$\|\chi_\omega \varphi_2(\cdot)\|=0$ over $(0,T)$. This, along with
the unique continuation of heat equations, gives that
$\varphi_2\equiv 0$ over $[0,T]$, which contradicts with the
fact that $\chi_\omega\varphi_1\neq \chi_\omega\varphi_2$. Hence,
(\ref{proof-J-T-y0-unique-not-0}) holds.

Two observations are given in order: First, it is clear that
\begin{eqnarray}\label{proof-J-T-y0-unique-exp1}
 \left(\int_0^T\|\chi_\omega\varphi_{\lambda}\|\,\mathrm{d}t\right)^2
\leq
\left((1-\lambda)\int_0^T\|\chi_\omega\varphi_1\|\,\mathrm{d}t
+\lambda\int_0^T\|\chi_\omega\varphi_2\|\,\mathrm{d}t\right)^2.
\end{eqnarray}
Since  $\varphi_1, \varphi_2, \varphi_\lambda\in C([0,T);
L^2(\Omega))$ (see Lemma~\ref{wanglemma4.3})) and because
\begin{equation}\label{proof-J-T-y0-unique-exp2}
\|(1-\lambda)\chi_\omega\varphi_1(t)
+\lambda\chi_\omega\varphi_2(t)\|\leq
(1-\lambda)\|\chi_\omega\varphi_1(t)\|
+\lambda\|\chi_\omega\varphi_2(t)\|\textrm{ for each }t\in [0,T),
\end{equation}
the equality in (\ref{proof-J-T-y0-unique-exp1}) holds if and only
if  the equality in (\ref{proof-J-T-y0-unique-exp2}) holds for each
$t\in [0,T)$. On the other hand,  the equality in
(\ref{proof-J-T-y0-unique-exp2}) holds for each $t\in [0,T)$ if and
only if  for each $t\in [0,T)$, there is a $d(t)>0$ such that
\begin{equation}\label{proof-J-T-y0-unique-d}
 \chi_\omega\varphi_1(t)=d(t){\chi_{\omega}}\varphi_2(t)\textrm{ in }L^2(\Omega).
\end{equation}
Thus, the equality in (\ref{proof-J-T-y0-unique-exp1}) holds if and
only if (\ref{proof-J-T-y0-unique-d}) stands.  Second, it is obvious
that
\begin{eqnarray}\label{proof-J-T-y0-unique-exp3}
&&\Big((1-\lambda)\int_0^T\|\chi_\omega\varphi_1\|\,\mathrm{d}t
+\lambda\int_0^T\|\chi_\omega\varphi_2\|\,\mathrm{d}t\Big)^2\nonumber\\
&\leq&(1-\lambda)\Big(\int_0^T\|\chi_\omega\varphi_1\|\,\mathrm{d}t\Big)^2
+\lambda\Big(\int_0^T\|\chi_\omega\varphi_2\|\,\mathrm{d}t\Big)^2
\end{eqnarray}
and the equality in (\ref{proof-J-T-y0-unique-exp3}) holds if and
only if
\begin{equation}\label{proof-J-T-y0-unique-exp4}
\int_0^T\|\chi_\omega\varphi_1\|\,\mathrm{d}t=\int_0^T\|\chi_\omega\varphi_2\|\,\mathrm{d}t.
\end{equation}
By (\ref{proof-J-T-y0-unique-con}), we see that the equalities in
both (\ref{proof-J-T-y0-unique-exp1}) and
(\ref{proof-J-T-y0-unique-exp3}) hold respectively. Hence, we have
both (\ref{proof-J-T-y0-unique-d}) and
(\ref{proof-J-T-y0-unique-exp4}). Since
$\|\chi_\omega\varphi_2(t)\|\neq 0$ for each $t\in[0,T)$ (see
(\ref{proof-J-T-y0-unique-not-0})), we derive from
(\ref{proof-J-T-y0-unique-d}) that
$d(t)={\|\chi_\omega\varphi_1(t)\|}/{\|\chi_\omega\varphi_2(t)\|}$
for each $t\in[0,T)$. This, along with  the fact that  $\varphi_1, \varphi_2\in C([0,T); L^2(\Omega))$,
indicates that
 $d(\cdot)\in
C([0,T);\mathbb R^+)$. By making use of
(\ref{proof-J-T-y0-unique-d}) again, we find that
\begin{equation}\label{proof-J-T-y0-unique-d-int}
\int_0^Td(t)\|\chi_\omega\varphi_2(t)\|\,\mathrm{d}t
=\int_0^T\|\chi_\omega\varphi_1\|\,\mathrm{d}t.
\end{equation}
Applying the mean value theorem of integral to the left side of
(\ref{proof-J-T-y0-unique-d-int}), we get that there is a $\hat
t\in(0,T)$ such that
\begin{equation}\label{proof-J-T-y0-unique-d-hat-t}
\int_0^Td(t)\|\chi_\omega\varphi_2(t)\|\,\mathrm{d}t=d(\hat
t)\int_0^T\|\chi_\omega\varphi_2\|\,\mathrm{d}t.
\end{equation}
From (\ref{proof-J-T-y0-unique-exp4}),
(\ref{proof-J-T-y0-unique-d-int}) and
(\ref{proof-J-T-y0-unique-d-hat-t}), it follows that $d(\hat t)=1$.
This, as well as   (\ref{proof-J-T-y0-unique-d}), leads to
$\chi_\omega \varphi_1(\hat t)=\chi_\omega \varphi_2(\hat t)$ in
$L^2(\Omega)$, which, together with the unique continuation for heat
equations, yields that $\varphi_1=\varphi_2$ over
$[0,T]$. This leads to a contradiction. Hence, $J^{T,1}_{y_0}$ is
strictly convex in $Y_{T,1}$.

$(ii)$ The proof follows from the same way used to prove the part
$(i)$ of Lemma~\ref{norm-Lp}. This completes the proof.

\end{proof}

The proof of the existence  for the minimizer to
$J^{T,1}_{y_0}$ (under  the assumption  (\ref{Y-Z})), as well as
of Theorem~\ref{theorem-bang-bang}, needs the help of the following
preliminaries. Let
\begin{equation}\label{beta-t-T}
\beta(t,T)\triangleq \sup_{z\in
L^2(\Omega)\backslash\{0\}}\frac{\|\varphi(t;T,z)\|}{\|\chi_\omega\varphi(\cdot;T,z)\|_{L^1(t,T;
L^2(\Omega))}}, \,T>0,\,t\in[0,T).
\end{equation}
The term on the right hand side of (\ref{beta-t-T}) is well-defined
because of the unique continuation for heat equations. From
Proposition 3.2 in \cite{FZ}, we can derive the following estimate:
\begin{equation}\label{beta-t-T-C1}
\beta(t, T)\leq C_1(T,t),\, \textrm{ for all } T>0,\; t\in[0,T).
\end{equation}
Here,
\begin{equation}\label{C1}
C_1(T,
t)\triangleq\textrm{exp}\Big[\Big(1+\frac{1}{T-t}\Big)\widehat
C_0\Big],\, T>0,\,t\in[0,T),
\end{equation}
where $\widehat C_0>0$ depends only on $\Omega$, $\omega$ and
$\|a\|_\infty$ which is the
$L^\infty(\Omega\times\mathbb{R}^+)$-norm of $a$. The proof of
(\ref{beta-t-T-C1}) will be given in Appendix, for sake of the
completeness  of the paper.

\begin{Lemma}\label{existence-J}
 Let $T>0$ and $y_0\in L^2(\Omega\setminus\{0\}$. Suppose that (\ref{Y-Z}) holds. Then $J_{y_0}^{T,1}$
 has  a  minimizer $\chi_\omega\widehat\varphi$ in $Y_{T,1}$.
 Furthermore, it holds that
 \begin{equation}\label{JP-value}
  V_1(T,y_0)=-\frac{1}{2} \Big(\int_0^T \|\chi_\omega\widehat\varphi\| \;\mathrm{d}t \Big)^2,
 \end{equation}
 where $V_1(T,y_0)$ is given by (\ref{J-y0}).
\end{Lemma}

\begin{proof}

 We start with proving the coercivity of $J^{T,1}_{y_0}$.
By   Lemma~\ref{wanglemma4.3}, and by using  the standard density
argument,   one can easily derive from
 (\ref{beta-t-T}) and (\ref{beta-t-T-C1}) that
\begin{equation}\label{varphi-t-z-C1}
\|\varphi(t)\|\leq C_1(T,
t)\int_t^T\|\chi_\omega\varphi\|\,\mathrm{d}s\;\;\mbox{for
all}\;\; T>0, t\in[0,T)\;\;\mbox{and}\;\;\chi_\omega\varphi\in
Y_{T,1}.
\end{equation}
From (\ref{varphi-t-z-C1}), we see that
\begin{eqnarray*}
 \langle y_0,\varphi(0) \rangle
   \geq - C_1(T, 0)^2\|y_0\|^2 - \frac{1}{4} \Big(\int_0^T \|\chi_\omega\varphi\| \,\mathrm{d}t \Big)^2\;\;\mbox{for each}\;\;\chi_\omega\varphi\in Y_{T,1}.
\end{eqnarray*}
This, along with (\ref{J-y0}) and (\ref{WANG1.6}), indicates that
\begin{eqnarray*}
 J^{T,1}_{y_0}(\chi_\omega\varphi)\geq \frac{1}{4}\|\chi_\omega\varphi\|_{Y_{T,1}}^2 - C_1(T, 0)^2\|y_0\|^2\;\;\mbox{for each}\;\;\chi_\omega\varphi\in Y_{T,1},
\end{eqnarray*}
which leads to the coercivity of  $J_{y_0}^{T,1}$.

We next write $\{\chi_\omega\varphi_{n}\}$ for a minimizing sequence
of $J^{T,1}_{y_0}$. By  the coercivity of $J^{T,1}_{y_0}$, there is
a positive constant $C$ independent of $n$ such that
\begin{equation}\label{JP-1}
 \int_0^T \|\chi_\omega\varphi_n\|\;\mathrm{d}t \leq C\;\;\mbox{for all}\;\; n\in \mathbb{N}.
\end{equation}
Let $\{T_k\}\subset (0, T)$ be such that
 $T_k\nearrow T$. By
(\ref{varphi-t-z-C1}) and (\ref{JP-1}),  it holds that
\begin{eqnarray}\label{JP-2}
  \|\varphi_n(T_k)\|  \leq C_1(T, T_k) \int_{T_k}^{T}\|\chi_\omega\varphi_n\|\;\mathrm{d}t \leq CC_1(T, T_k) \triangleq C(k), \forall\;n, k\in
  \mathbb{N}
\end{eqnarray}
Let $k=2$ in (\ref{JP-2}). By properties of heat equations, there
are a $z_1\in L^2(\Omega)$ and
 a subsequence $\{\varphi_{n_l}\}$ of
$\{\varphi_n\}$ such that
\begin{equation*}
 \varphi_{n_l}(\cdot)\rightarrow \varphi(\cdot; z_1, T_1)\;\;\mbox{strongly in}\;\;C([0,T_1];L^2(\Omega)),\;\;\mbox{as}\;\;l \rightarrow \infty.
\end{equation*}
Let $k=3$ in (\ref{JP-2}). By properties of heat equations, we can
find a $z_2\in L^2(\Omega)$ and
 a subsequence $\{\varphi_{n_{l_s}}\}$
of $\{\varphi_{n_l}\}$   such that
\begin{equation*}
 \varphi_{n_{l_s}}(\cdot)\rightarrow \varphi(\cdot; z_{2}, T_{2})\;\;\mbox{strongly in}\;\;C([0,T_{2}];L^2(\Omega)),\;\;\mbox{as}\;\;s \rightarrow \infty.
\end{equation*}
Continuing this procedure with respect to $k$, and then using the
diagonal law, we find a subsequence of $\{\varphi_n\}$, still
denoted in the same way, and a sequence $\{z_k\}$ in $L^2(\Omega)$
such that for each $k\geq 2$,
\begin{equation}\label{JP-converge-k}
\varphi_{n}(\cdot)\rightarrow \varphi(\cdot; z_k,
T_k)\;\;\mbox{strongly in}\;\;C([0,T_k];L^2(\Omega)),~~\textrm{ as }
n\rightarrow \infty.
\end{equation}
From (\ref{JP-converge-k}), we see that
\begin{equation}\label{JP-converge-well}
 \varphi(t; z_k,T_k)=\varphi(t; z_{k+j}, T_{k+j})\;\;\mbox{for all}\;\; t\in [0, T_k], ~k=2,3,\dots, ~ j=1,2,\dots.
\end{equation}
Now, we define a function $\hat \varphi$ over $[0, T)$ by setting
\begin{equation}\label{JP-converge-limit}
\hat \varphi(t)=\varphi(t; z_k, T_k),\;\; t\in [0,T_k],~k=2,3,\dots.
\end{equation}
From this and (\ref{JP-converge-well}),  $\hat\varphi$ is well-defined. Then by
(\ref{JP-converge-k}) and (\ref{JP-converge-limit}), we see that
\begin{equation}\label{JP-converge-property1}
 \hat \varphi\in C([0,T);L^2(\Omega)) \mbox{ solves  Equation
 (\ref{adjoint-equation})}
\end{equation}
and
\begin{equation}\label{wanG3.41}
 \chi_\omega\varphi_{n}\rightarrow \chi_\omega\hat \varphi\;\;\mbox{strongly in}\;\;L^1(0,T_k; L^2(\omega))~~\textrm{ for each } k.
\end{equation}
From (\ref{wanG3.41}) and (\ref{JP-1}), we find that for each $k\in
\mathbb{N}$,
\begin{equation*}
 \int_0^{T_k} \|\chi_\omega \hat \varphi\| \;\textrm{d}t = \liminf_{n\rightarrow \infty} \int_0^{T_k}
 \|\chi_\omega \varphi_n\| \;\textrm{d}t \leq \liminf_{n\rightarrow \infty}
 \int_0^{T} \|\chi_\omega \varphi_n\| \;\textrm{d}t\leq C.
\end{equation*}
This implies
\begin{equation}\label{JP-converge-property2}
 \int_0^{T} \|\chi_\omega \hat\varphi\| \;\textrm{d}t  \leq \liminf_{n\rightarrow \infty}
 \int_0^{T} \|\chi_\omega \varphi_n\| \;\textrm{d}t\leq C.
\end{equation}
From (\ref{zT1}), (\ref{JP-converge-property1}) and
(\ref{JP-converge-property2}), it follows that
$\chi_\omega\hat\varphi\in Z_{T,1}$. This, along with the assumption
(\ref{Y-Z}), indicates that
\begin{equation}\label{waNg3.43}
\chi_\omega\hat \varphi\in Y_{T,1}.
\end{equation}
From (\ref{JP-converge-k}) and (\ref{JP-converge-limit}), we, in
particular, have that $\varphi_n(0)\rightarrow \hat\varphi (0)$
strongly in $L^2(\Omega)$. This, together with (\ref{J-y0})  and
(\ref{JP-converge-property2}), yields
\begin{equation}\label{JP-3}
 J^{T,1}_{y_0}(\chi_\omega \hat\varphi) \leq \displaystyle\liminf_{n\rightarrow\infty} J^{T,1}_{y_0}(\chi_\omega\varphi_n).
\end{equation}
From (\ref{JP-3}) and (\ref{waNg3.43}), we see  that
$\chi_\omega\hat\varphi$ is the minimizer of $J^{T,1}_{y_0}(\cdot)$.

Finally, we prove (\ref{JP-value}). The Euler-Lagrange equation
associated with  $\chi_\omega\hat\varphi$ reads:
\begin{equation}\label{JP-Euler-equation}
\langle y_0,\varphi(0)\rangle+\int_0^T\langle \hat u(t),\chi_\omega
\varphi(t)\rangle\,\mathrm{d}t=0,\, \chi_\omega\varphi\in Y_{T,1} ,
\end{equation}
where
\begin{equation*}
\hat u(t)=\int_0^T\|\chi_\omega\widehat \varphi\|\,\mathrm{d}s
\cdot \frac{\chi_\omega\widehat \varphi(t)}{\|\chi_\omega\widehat
\varphi(t)\|},\, t\in [0,T).
\end{equation*}
Letting  $\varphi=\widehat\varphi$ in (\ref{JP-Euler-equation}), we
get
\begin{equation*}
 \langle y_0,\widehat\varphi(0)\rangle+\Big(\int_0^T \|\chi_\omega \widehat\varphi\|\,\mathrm{d}t \Big)^2=0.
\end{equation*}
This, along with (\ref{J-y0}), leads to (\ref{JP-value}) and
completes the proof.

\end{proof}

\section{The bang-bang property  for $(TP)^{M,p}_{y_0}$}

This section is mainly devoted to the proof of Theorem
\ref{theorem-bang-bang}. Our strategy  is as follows. We first show
that $(TP)^{M,p}_{y_0}$ has the bang-bang property if and only if
$M=N_p(T, y_0)$ for some $T>0$; then prove that the function
$N_p(\cdot, y_0)$ is strictly monotonically decreasing and
continuous from $(0, \infty)$ onto $(\widehat N_p(y_0), \infty)$;
finally, through utilizing the bang-bang property of
$(NP)^{T,p}_{y_0}$ (see Lemma~\ref{lemma-NP-control}), derive the
bang-bang property for $(TP)^{M,p}_{y_0}$ for any $M>\widehat N_p(y_0)$. To show the
left continuity of $N_p(\cdot, y_0)$, we need  the assumption
(\ref{Y-Z}).

\begin{Lemma}\label{lemma-bang-bang}
 Let $y_0\in L^2(\Omega)\setminus \{0\}$. Then $(TP)_{y_0}^{M,p}$, with $M>0$,  has the bang-bang property if and only if $M=N_p(T,y_0)$ for some $T>0$.
\end{Lemma}

\begin{proof}
 First we suppose that  $M=N_p(T,y_0)$ for some $T>0$.
Let $u_1$ be the optimal
 control to  $(NP)_{y_0}^{T,p}$. (The existence of the optimal control
is ensured by  Lemma
 \ref{lemma-NP-control}.)
One can easily check that $u_1$ is an admissible control to
$(TP)^{M,p}_{y_0}$. This,  along with the definition of $T_p(M,y_0)$
(see (\ref{TP})), yields that
 \begin{equation}\label{NP-4-2}
  T_p(M,y_0)\leq T.
 \end{equation}
Meanwhile, since $(TP)^{M,p}_{y_0}$ has admissible controls, one can
use the standard way  to show that $(TP)^{M,p}_{y_0}$ has  optimal
controls (see for instance, \cite{HOF}, or  the proof of Lemma 3.2
in \cite{PWZ}). Arbitrarily take an optimal control  $u_2$ to
$(TP)^{M,p}_{y_0}$. Clearly,
 \begin{equation}\label{NP-4-3}
  \|u_2\|_{L^p(\mathbb R^+;L^2(\Omega))}\leq M=N_p(T,y_0)\;\;\mbox{and}\;\;y(T_p(M,y_0);y_0,u_2)=0.
 \end{equation}
Let  $\widehat u_2\in L^p(0,T; L^2(\Omega))$ be such that $\widehat
u_2=u_2$ over $[0, T_p(M,y_0))$ and $\widehat
u_2=0$ over $[T_p(M,y_0), T]$. From
(\ref{NP-4-3}) and (\ref{NP-4-2}), one can easily verify that
 $\widehat u_2$ is an optimal control to
$(NP)^{T,p}_{y_0}$. Since $\widehat u_2(t)=0$ for a.e. $ t\in
[T_p(M,y_0),T]$, it follows from the bang-bang property of
$(NP)_{y_0}^{T,p}$ (see  Lemma \ref{lemma-NP-control}) that
$T_p(M,y_0)=T$ and $u_2=\widehat u_2$ over $(0, T)$.
These, along with the bang-bang property of $(NP)_{y_0}^{T,p}$ (see
Definition~\ref{definition1.2}), lead to the bang-bang property of
$(TP)^{M,p}_{y_0}$ (see Definition~\ref{definition 1.1}).

Conversely, we suppose that for some $M>0$,   $(TP)^{M,p}_{y_0}$ has
the bang-bang property.
  It suffices to show
  \begin{equation}\label{NP-4-5}
     M=N_p(T_p(M,y_0),y_0).
  \end{equation}
 From Remark~\ref{remarkwang1.1}, $(TP)^{M,p}_{y_0}$ has an optimal
 control $u_3$, which is clearly an admissible control to
 $(NP)^{T_p(M,y_0),p}_{y_0}$.
Thus, it holds that
  \begin{equation}\label{NP-4-8}
   N_p(T_p(M,y_0),y_0)\leq \|u_3\|_{L^p(0, T_p(M,y_0); L^2(\Omega))} \leq M.
  \end{equation}
Let $u_4$ be the optimal control to $(NP)^{T_p(M,y_0),p}_{y_0}$.
Then
  \begin{equation}\label{NP-4-6}
   y(T_p(M,y_0);y_0,u_4)=0\;\;\mbox{and}\;\;\|u_4\|_{L^p(0,T_p(M,y_0);L^2(\Omega))}=N_p(T_p(M,y_0),y_0).
  \end{equation}
We extend $u_4$ over $\mathbb{R}^+$ by  setting it to be $0$
over  $[T_p(M,y_0), \infty)$, and denote the extension by $\widehat
u_4$. Then, from  (\ref{NP-4-6}) and (\ref{NP-4-8}),
 we see that $\widehat u_4$ is an optimal control to
 $(TP)^{M,p}_{y_0}$. By the bang-bang property of  $(TP)^{M,p}_{y_0}$ (see Definition~\ref{definition 1.1}), we find that
$$
\|\chi_\omega u_4\|_{L^p(0,T_p(M,y_0);
L^2(\Omega))}=\|\chi_\omega \widehat u_4\|_{L^p(0,
T_p(M,y_0); L^2(\Omega))}=M,
$$
which, along with  (\ref{NP-4-6}) and (\ref{NP-4-8}), leads to
(\ref{NP-4-5}). This completes the proof.

\end{proof}

\begin{Lemma}\label{lemma-N-t-lim}
 Let $y_0\in L^2(\Omega)\setminus \{0\}$. $(i)$ The function $N_p(\cdot,y_0)$ is strictly monotonically decreasing and right-continuous over $(0, \infty)$.
 Moreover, it holds that
\begin{equation}\label{N-t-lim-0}
\lim_{T\to 0^+}N_p(T,y_0)=\infty
\end{equation}
and
\begin{equation}\label{N-t-lim-inf}
\lim_{T\to+\infty}N_p(T,y_0) = \widehat N_p(y_0)\in[0,
\infty),
\end{equation}
 where $\widehat N_p(y_0)$ is given by
(\ref{HUANG1.12}); $(ii)$ Suppose that (\ref{Y-Z}) holds. Then the
function $N_p(\cdot,y_0)$ is left-continuous from { $(0, \infty)$
onto $(\widehat N_p(y_0), \infty)$}.
\end{Lemma}

\begin{proof}
$(i)$  We start with showing the strictly monotonicity of
$N_p(\cdot,y_0)$. Let $0<T_1<T_2$. Let $u_1$ be the optimal control
to $(NP)_{y_0}^{T_1,p}$. We extend $u_1$ over $(0, T_2)$ by
setting it to be $0$ over $(T_1, T_2)$ and denote the extension by $u_2$. It is clear that
\begin{equation}\label{proof-N-t-lim-u2-t2}
y(T_2; y_0,u_2)=0.
\end{equation}
Hence, $u_2$ is an admissible control to $(NP)_{y_0}^{T_2,p}$.
Therefore, it holds that
\begin{equation}\label{yuanyuan4.9}
N_p(T_1,y_0)=\|u_1\|_{L^p(0,T_1;L^2(\Omega))}=\|u_2\|_{L^p(0,T_2;L^2(\Omega))}\geq
N_p(T_2,y_0).
\end{equation}
\noindent We claim that $N_p(T_1,y_0)>N_p(T_2,y_0)$.
 By
contradiction, we suppose that it did not hold. Then by
(\ref{yuanyuan4.9}), we would have $N_p(T_1,y_0)=N_p(T_2,y_0)$.
Thus,
$$
\|u_2\|_{L^p(0,T_2;L^2(\Omega))}=\|u_1\|_{{L^p(0,T_1;L^2(\Omega))}}
=N_p(T_1,y_0)=N_p(T_2,y_0).
$$
This, together with (\ref{proof-N-t-lim-u2-t2}), shows that $u_2$ is
an optimal control to $(NP)_{y_0}^{T_2,p}$. By the bang-bang
property of $(NP)_{y_0}^{T_2,p}$ (see Lemma \ref{lemma-NP-control}),
we have that $\|\chi_\omega u_2(t)\|\neq 0$ for a.e. $t\in (0, T_2)$
(see Definition~\ref{definition1.2}). This contradicts with the fact
that $u_2=0$ over $(T_1, T_2)$. Hence, $N_p(\cdot, y_0)$ is
strictly monotonically decreasing.

Next, we show the right-continuity of $N_p(\cdot,y_0)$. Arbitrarily
fix a $\widehat T\in {(0,\infty)}$. Let $\{T_n\}\subset (\widehat
T,\widehat T+1)$ be such that $T_n\searrow {\widehat T}$. Then by
the monotonicity of $N_p(\cdot,y_0)$, there is a $\widehat M\in (0,
\infty)$ such that
\begin{equation}\label{proof-theo-mono2-tilde-M}
N_p(T_n,y_0)\nearrow \widehat M.
\end{equation}
It suffices to show
\begin{equation}\label{proof-theo-mono2-tilde-M-N}
\widehat M=N_p(\widehat T,y_0).
\end{equation}
Seeking for a contradiction, we suppose that
(\ref{proof-theo-mono2-tilde-M-N}) did not hold. Then by the
monotonicity of  $N_p(\cdot,y_0)$, we would have
\begin{equation}\label{proof-theo-mono2-tilde-M-con}
\widehat M<N_p(\widehat T,y_0).
\end{equation}
Let $u_n$ be the optimal control to $(NP)_{y_0}^{T_n,p}$. We extend
$u_n$ over $(0, \widehat T+1)$ by setting it to be $0$ over $(T_n,
\widehat T+1)$, and denote the extension by $\hat u_n$. Then one can
easily check that
\begin{equation}\label{proof-theo-mono2-hat-un}
\|\hat u_n\|_{L^p(0,\widehat T+1;L^2(\Omega))}=N_p(T_n,y_0)\leq
\widehat M\;\;\mbox{and}\;\; y(T_n;y_0,\hat u_n)=0.
\end{equation}
Thus, we can extract a subsequence from $\{\hat u_n\}$, still
denoted in the same way, such that for some $\hat u\in
L^p(0,\widehat T+1;L^2(\Omega))$,
\begin{equation}\label{proof-theo-mono2-hat-un-to-hat-u}
\widehat u_n\to \hat u\textrm{ weakly star in }L^p(0,\widehat T+1;L^2(\Omega)).
\end{equation}
This, along with (\ref{proof-theo-mono2-hat-un}) and
(\ref{proof-theo-mono2-tilde-M-con}), yields
\begin{equation}\label{proof-theo-mono2-hat-u-n}
\|\hat u\|_{L^p(0,\widehat T+1;L^2(\Omega))}\leq
\liminf_{n\to\infty}\|u_n\|_{L^p(0,\widehat
T+1;L^2(\Omega))}\leq\widehat M<N_p(\widehat T,y_0).
\end{equation}
Meanwhile, by (\ref{proof-theo-mono2-hat-un-to-hat-u}) and the equations satisfied by $y(\cdot;y_0,\hat u_n)$ and $y(\cdot;y_0,\hat u)$ over
 $(0,\widehat T+1)$, using the standard argument involving the Ascoli-Arzel\`{a} theorem, we can get a subsequence of $\{y(\cdot;y_0,\hat u_n)\}$,
 denoted in the same way, such that
$$
y(\cdot;y_0,\hat u_n)\to y(\cdot;y_0,\hat u)\textrm{ in }C([0,\widehat T+1];L^2(\Omega)).
$$
This, together with (\ref{proof-theo-mono2-hat-u-n}) and the second equality in (\ref{proof-theo-mono2-hat-un}), indicates that
\begin{eqnarray*}
  \|y(\widehat T;y_0,\hat u)\| &\leq & \|y(\widehat T;y_0,\hat u)-y(T_n;y_0,\hat u)\|
  + \|y(T_n;y_0,\hat u)-y(T_n;y_0,\hat u_n)\|\\
  &&+\|y(T_n;y_0,\hat u_n)\|\to 0,
\end{eqnarray*}
i.e.,  $y(\widehat T;y_0,\hat u)=0$. Thus, $\hat u$ is an admissible
control to $(NP)_{y_0}^{\widehat T,p}$, which yields
$$
\|\hat u\|_{L^p(0,\widehat T;L^2(\Omega))}\geq N_p(\widehat T,y_0).
$$
This contradicts with (\ref{proof-theo-mono2-hat-u-n}). Hence, $N_p(\cdot, y_0)$ is right continuous over {$(0,+\infty)$}.

Finally, we show (\ref{N-t-lim-inf}) and (\ref{N-t-lim-0}). Since
$N_p(T,y_0)>0$ for each $T>0$ (notice that $y_0\neq 0$),
(\ref{N-t-lim-inf}) follows from the monotonicity of
$N_p(\cdot,y_0)$ at once. To prove  (\ref{N-t-lim-0}), we suppose,
by contradiction  that it did not hold. Then there would be a
sequence
  $\{T_n\}\subset(0,1)$ such that $T_n\searrow0$ and $N_p(T_n,y_0)\nearrow\widetilde N\in (0, \infty)$. Let $u_n$ be the optimal
  control to $(NP)_{y_0}^{T_n,p}$. Then
\begin{equation}\label{proof-N-t-lim-un-n}
\|u_n\|_{L^p(0,T_n;L^2(\Omega))}=N_p(T_n,y_0)\leq \widetilde N
\quad\textrm{ for all } n\in \mathbb N^+
\end{equation}
and
\begin{equation}\label{proof-N-t-lim-un-tn}
0=y(T_n;y_0,u_n)=\Phi(T_n,0)y_0+\int_0^{T_n}\Phi(T_n,s)\chi_\omega
u_n(s)\,\mathrm{d}s,
\end{equation}
where $\{\Phi(t,s)\; |\; 0\leq s\leq t< \infty\}$ is the evolution
system generated by $\Delta- aI$ (see Chapter 5  in \cite{P}). By
(\ref{proof-N-t-lim-un-n}), we have
$$
\Big\|\int_0^{T_n}\Phi(T_n,s)\chi_\omega u_n(s)\,\mathrm{d}s\|\leq
\sup_{0\leq s\leq t\leq
1}\|\Phi(t,s)\Big\|_{\mathcal{L}(L^2(\Omega))} \widetilde N\cdot
T_n^{1-\frac{1}{p}}\to 0.
$$
This, along with (\ref{proof-N-t-lim-un-tn}), yields
$$
0=\lim_{n\to \infty}\Phi(T_n,0)y_0=y_0\neq 0,
$$
which leads to a contradiction. Hence, (\ref{N-t-lim-0}) holds. This
completes the proof of the part $(i)$.

$(ii)$ Arbitrarily fix a $\widehat T\in{(0, \infty)}$. Let
$\{T_n\}\subset [\hat T /2, \hat T)$ be such that $T_n\nearrow \hat
T$.  By the monotonicity of $N_p(\cdot, y_0)$, it suffices to show
that on a subsequence of $\{T_n\}$, denoted in the same way,
\begin{equation}\label{proof-theo-mono3-N-Tn-y0}
N_p(T_n,y_0)\to N_p(\widehat T,y_0) \mbox{ as } n\rightarrow\infty.
\end{equation}
By Lemmas~\ref{wanglemma4.2}, \ref{lemma-J-T-y0-unique},
\ref{existence-J} and Lemma~\ref{wanglemma4.3}, the functional
$J_{y_0}^{T_n,q}$ has a unique non-zero minimizer
$\chi_\omega\psi_n$ (on $Y_{T_n,q}$), where $\psi_n\in
C([0,T_n);L^2(\Omega))\cap L^q(0,T_n; L^2(\omega))$ solves Equation
(\ref{adjoint-equation}) with $T$ being replaced by $T_n$.  From
(\ref{V-T-y0-N-T-y0}), (\ref{JP-value-q}) and (\ref{JP-value}) (see
Lemmas~\ref{proposition-V-T-y0}, \ref{wanglemma4.2} and
\ref{existence-J} respectively), it holds that
\begin{equation}\label{proof-theo-mono3-N-Tn-y0-eqn}
0<\|\chi_\omega\psi_n\|_{L^q(0,T_n;L^2(\Omega))}=
N_p(T_n,y_0)\;\;\mbox{for all}\;\; n\in\mathbb N.
\end{equation}
Since $T_n<\hat T$, it follows from $(i)$ of Lemma \ref{lemma-N-t-lim} that
$N_p(T_n,y_0)\leq N_p({\widehat T}/{2},y_0)$ for all $n\in \mathbb
N$. This, as well as (\ref{proof-theo-mono3-N-Tn-y0-eqn}), yields
that
\begin{equation}\label{proof-theo-mono3-psin-L1-n}
\|\chi_\omega\psi_n\|_{L^q(0,T_n;L^2(\Omega))}\leq N_p({\widehat
T}/{2},y_0)\textrm{ for all }n\in \mathbb N.
\end{equation}
We extend $\psi_n$ over $[0, \hat T)$ by setting it to be zero over
$[T_n, \hat T)$, and denote the extension by  $\widetilde \psi_n$.
Then by  (\ref{varphi-t-z-C1}), (\ref{C1}) and
(\ref{proof-theo-mono3-psin-L1-n}), one has
\begin{eqnarray}\label{bdd-point}
  \|\widetilde\psi_n(T_2)\|&=&\|\psi_n(T_2)\|\leq C_1(T_n, T_2) \int_{T_2}^{T_n}\|\chi_\omega\psi_n\|\,\mathrm{d}s\nonumber\\
      &\leq& C_1(T_n, T_2)(T_n-T_2)^{1-\frac{1}{q}} \|\chi_\omega\psi_n\|_{L^q(0,T_n;L^2(\Omega))} \nonumber\\
   &\leq& \exp\big[{\widehat C_0(1+{1}/{(T_{3}-T_2)})}\big](\widehat T-T_2)^{1-\frac{1}{q}}N_p({\widehat T}/{2},y_0) \nonumber\\
  &\triangleq& C(T_2,T_3,\hat T) \, N_p({\widehat T}/{2},y_0) ~~\textrm{ for each }n \geq 3.
\end{eqnarray}
By (\ref{bdd-point}) and the properties of heat equations, there are
a subsequence $\{\widetilde\psi_{n_l}\}$ of $\{\widetilde\psi_n\}$
and a  $z_1\in L^2(\Omega)$ such that
\begin{equation*}
 \widetilde\psi_{n_l}(T_1)\rightarrow z_1\;\;\mbox{strongly in}\;\;L^2(\Omega),\;\;\mbox{as}\;\; l\rightarrow \infty
\end{equation*}
and
\begin{equation*}
 \widetilde\psi_{n_l}(\cdot)\rightarrow \varphi(\cdot; z_1, T_1)\;\;\mbox{strongly in}\;\;C([0,T_1];L^2(\Omega)),\;\;\mbox{as}\;\; l\rightarrow \infty,
\end{equation*}
where $\varphi(\cdot; z_1, T_1)$ is the solution of Equation
(\ref{adjoint-equation}) (where $T=T_1$), with $\varphi(T_1)=z_1$.
With respect to $\widetilde \psi_{n_l}(T_3)$, we can have a similar
estimate as (\ref{bdd-point}). Thus, we can take a subsequence
$\{\widetilde\psi_{n_{l_s}}\}$ from $\{\widetilde\psi_{n_l}\}$ and
get a $z_2\in L^2(\Omega)$ such that
\begin{equation*}
 \widetilde\psi_{n_{l_s}}(\cdot)\rightarrow \varphi(\cdot; z_{2}, T_{2})\;\;\mbox{strongly in}\;\;C([0,T_{2}];L^2(\Omega)),\;\;\mbox{as}\;\;s\rightarrow \infty.
\end{equation*}
Continuing this procedure and making use of the diagonal law, we can
get  a sequence $\{z_k\}$ in $L^2(\Omega)$ and  a
 subsequence of
$\{\widetilde\psi_n\}$, still denoted in the same way, such that
\begin{equation}\label{converge-k}
 \widetilde\psi_{n}(\cdot)\rightarrow \varphi(\cdot; z_k, T_k)\;\;\mbox{strongly in}\;\;C([0,T_k];L^2(\Omega))~~\textrm{ for each } k\in\mathbb N.
\end{equation}
This implies that
\begin{equation}\label{converge-well}
 \varphi(t; z_k,T_k)=\varphi(t; z_{k+j}, T_{k+j})\;\;\mbox{for all}\;\; t\in [0, T_k], ~k=1,2,\dots, ~ j=1,2,\dots.
\end{equation}
We  construct a function $\psi$ over $[0,T)$ by setting
\begin{equation}\label{converge-limit}
 \psi(t)=\varphi(t; z_k, T_k),\;\; t\in [0,T_k],~k=1,2,\dots.
\end{equation}
By (\ref{converge-limit}) and (\ref{converge-well}),  $\psi$ is a
well-defined function over $[0,T)$. From (\ref{converge-k}) and
(\ref{converge-limit}), it follows that
\begin{equation}\label{converge-property1}
 \psi\in C([0,\widehat T);L^2(\Omega))\; \mbox{ solves  Equation
 (\ref{adjoint-equation}), where}\;\; T=\hat T;
\end{equation}
\begin{equation}\label{yaunyuan4.28}
\psi_n(0)\rightarrow\psi (0)\;\;\mbox{strongly in}\;\; L^2(\Omega)
\end{equation}
and
\begin{equation}\label{yuanyuan4.29}
 \chi_\omega\widetilde\psi_{n}\rightarrow \chi_\omega\psi\;\;\mbox{strongly in}\;\;L^q(0,T_k;L^2(\Omega))~~\textrm{ for each } k.
\end{equation}
From (\ref{yuanyuan4.29}), we have
\begin{equation*}
 \|\chi_\omega \psi\|_{L^q(0,T_k;L^2(\Omega))}= \liminf_{n\rightarrow\infty}  \|\chi_\omega \widetilde\psi_n\|_{L^q(0,T_k;L^2(\Omega))}
  \leq \liminf_{n\rightarrow\infty}  \|\chi_\omega \psi_n\|_{L^q(0,T_n;L^2(\Omega))},~\forall\, k\in\mathbb
  N^+,
\end{equation*}
which, along with (\ref{proof-theo-mono3-psin-L1-n}), yields that
\begin{equation}\label{converge-property2}
 \|\chi_\omega \psi\|_{L^q(0,\widehat T;L^2(\Omega))} \leq \liminf_{n\rightarrow\infty} \|\chi_\omega \psi_n\|_{L^q(0,T_n;L^2(\Omega))} \leq N_p({\widehat T}/{2},y_0).
\end{equation}
From (\ref{zT1}), (\ref{converge-property1}),
(\ref{converge-property2}) and (\ref{Y-Z}), we see that
\begin{equation}\label{yuanyuan3.31}
\chi_\omega\psi\in Z_{\widehat T,q}=Y_{\widehat T,q}.
\end{equation}
By (\ref{J-y0}), (\ref{yuanyuan3.31}),  (\ref{converge-property2})
and  (\ref{yaunyuan4.28}), one can easily verify that
\begin{equation*}\label{J-y0-limit}
 J^{\widehat T,q}_{y_0} (\chi_\omega\psi) \leq \liminf_{n\rightarrow\infty} J^{T_n,q}_{y_0}(\chi_\omega\psi_n)=
 \liminf_{n\rightarrow\infty} V_q(T_n,y_0).
\end{equation*}
This, along with (\ref{V-T-y0-N-T-y0}) (see
Lemma~\ref{proposition-V-T-y0}), indicates that
\begin{equation}\label{J-y0-limit-1}
 J^{\widehat T,q}_{y_0} (\chi_\omega\psi) \leq \liminf_{n\rightarrow\infty} -\frac{1}{2} N_p(T_n,y_0)^2.
\end{equation}
By (\ref{V-T-y0-N-T-y0}), (\ref{yuanyuan3.31}), (\ref{J-y0}) and
(\ref{J-y0-limit-1}), we see that
\begin{equation*}
 -\frac{1}{2} N_p(\widehat T,y_0)^2=V_q(\widehat T,y_0)\leq J^{\widehat T,q}_{y_0} (\chi_\omega\psi) \leq \liminf_{n\rightarrow\infty}
  -\frac{1}{2} N_p(T_n,y_0)^2,
\end{equation*}
from which, it follows that
\begin{equation}\label{norm-left}
 \limsup_{n\rightarrow\infty} N_p(T_n,y_0) \leq N_p(\widehat T,y_0).
\end{equation}
On the other hand, since $N_p(\cdot,y_0)$ is decreasing and $T_n<
\widehat T$ for all $n$, it holds that
\begin{equation}\label{yuanyuan4.34}
 \liminf_{n\rightarrow\infty} N_p(T_n,y_0) \geq N_p(\widehat T,y_0).
\end{equation}
Now, (\ref{proof-theo-mono3-N-Tn-y0}) follows from (\ref{norm-left})
and (\ref{yuanyuan4.34}) at once. This completes the proof.

\end{proof}

With the aid of the part $(i)$ of Lemma~\ref{lemma-N-t-lim}, we can
prove the following existence result on optimal controls to Problem
$(TP)^{M,p}_{y_0}$.

\begin{Proposition}\label{proposition-TP-M-y0}
Let $y_0\in L^2(\Omega)\setminus\{0\}$. Then problem
$(TP)_{y_0}^{M,p}$, with $M>0$, has optimal controls iff $M\in
(\widehat N_p(y_0),\infty)$ where $\widehat N_p(y_0)$ is given by
(\ref{HUANG1.12}).
\end{Proposition}
\begin{proof}
First we suppose that $M\in(\widehat N_p(y_0), \infty)$. Then by
(\ref{N-t-lim-inf}) and the monotonicity of $N_p(\cdot,y_0)$ (see
the part $(i)$ of  Lemma \ref{lemma-N-t-lim}), there is a $T_1\in(0,
\infty)$ such that $N_p(T_1,y_0)<M$. Let $u_1$ be the optimal
control to $(NP)_{y_0}^{T_1,p}$. (The existence of optimal controls
is ensured by  Lemma \ref{lemma-NP-control}). Then we have
\begin{equation*}
\|u_1\|_{L^p(0,T_1;L^2(\Omega))}=N_p(T_1,y_0)<M\;\;\mbox{and}\;\;y(T_1;y_0,u_1)=0.
\end{equation*}
From these , $u_1$ is an admissible control to $(TP)_{y_0}^{M,p}$.
By the standard arguments (see, for instance, the proof of Lemma 3.2
in \cite{PWZ}), we can get the existence of optimal controls to
$(TP)_{y_0}^{M,p}$.

Conversely, we assume that $M\leq \widehat N_p(y_0)$. Seeking for a
contradiction, we suppose that $(TP)_{y_0}^{M,p}$ did have an
optimal control $\bar u$ in this case. Then we would have that
\begin{equation}\label{proof-propo-TP-bar-u}
\|\bar u\|_{L^p(0,T_p(M,y_0);L^2(\Omega))}\leq M
\end{equation}
and
\begin{equation}\label{proof-propo-TP-bar-u-y-0}
y(T_p(M,y_0);y_0,\bar u)=0.
\end{equation}
By (\ref{proof-propo-TP-bar-u-y-0}), $\bar u$ is an admissible
control to $(NP)_{y_0}^{T_p(M,y_0),p}$. Then by
(\ref{proof-propo-TP-bar-u}) and the optimality of
$N_p(T_p(M,y_0),y_0)$, it holds that $N_p(T_p(M,y_0),y_0)\leq M$.
This, along with the strict monotonicity of $N_p(\cdot, y_0)$ (see
Lemma \ref{lemma-N-t-lim}), yields that $ M\geq
N_p(T_p(M,y_0),y_0)>\widehat N_p(y_0)$, which leads to a
contradiction. This completes the proof.
\end{proof}

Now we prove Theorem~\ref{theorem-bang-bang}.

\begin{proof}[Proof of Theorem~\ref{theorem-bang-bang}]
When $M\leq \widehat N_p(y_0)$, it follows from
Proposition~\ref{proposition-TP-M-y0} that $(TP)_{y_0}^{M,p}$ has no
any optimal control. Hence,  it has no bang-bang property (see
Remark~\ref{remarkwang1.1}). Conversely, if  $ M>\widehat
N_p(y_0)$, then by  Lemma \ref{lemma-N-t-lim},
 there is a unique $\widetilde T\in(0, \infty)$ such that $ M=N_p(\widetilde T,y_0)$. According to  Lemma \ref{lemma-bang-bang},
  $(TP)_{y_0}^{ M,p}$ has the bang-bang
 property. This completes the proof of
 Theorem~\ref{theorem-bang-bang}.

\end{proof}

Finally, we will show that the condition (\ref{Y-Z}) holds for some
cases.

\begin{Proposition}\label{lemma-YT-ZT}
Suppose that  $a\in L^\infty(\Omega\times \mathbb{R}^+)$ verifies
$a(x,t)=a_1(x)+a_2(t)$ in $\Omega\times\mathbb R^+$, with $a_1\in
L^\infty(\Omega)$ and $a_2\in L^\infty(\mathbb R^+)$. Then
$Y_{T,q}=Z_{T,q}$ for all $T>0$ and $q\in [1, \infty)$.
\end{Proposition}

\begin{proof}
It suffices to show that
\begin{equation}\label{YZ-1}
 Z_{T,q} \subset Y_{T,q},\;\;\mbox{when}\;\; T>0\;\;\mbox{and}\;\; q\in [1, \infty).
\end{equation}
Let $T>0$ and $q\in [1,\infty)$ be arbitrarily given. Observe that $
\psi\in C([0,T);L^2(\Omega))\cap L^q(0,T;L^2(\omega))$ solves the
equation:
\begin{eqnarray}\label{dualequation2}
\left\{
   \begin{array}{lll}
   \partial_t \psi(x,t)+\Delta \psi(x,t) -(a_1(x)+a_2(t))\psi(x,t)=0  &\mbox{ in }  &\Omega\times(0,T),\\
   \psi(x,t)=0   &\mbox{ on }   &\partial\Omega\times(0,T)
   \end{array}
\right.
\end{eqnarray}
if and only if $\varphi\in C([0,T);L^2(\Omega))\cap
L^q(0,T;L^2(\omega))$ solves
\begin{eqnarray}\label{dualequation2-1}
\left\{
   \begin{array}{lll}
   \partial_t \varphi(x,t)+\Delta \varphi(x,t) -a_1(x)\varphi(x,t)=0  &\mbox{ in }  &\Omega\times(0,T),\\
   \varphi(x,t)=0   &\mbox{ on }   &\partial\Omega\times(0,T),
   \end{array}
\right.
\end{eqnarray}
where the function $\varphi$ is defined by
\begin{equation}\label{Y-Z-11}
\varphi(x,t)=\exp \Big[{\int_t^T a_2(\tau) \mathrm{d}\tau}\Big]
\psi(x,t),\; (x,t)\in \Omega\times (0,T).
\end{equation}

Given $\chi_\omega\hat \psi\in Z_{T,q}$, let $\hat\varphi$ be given
by (\ref{Y-Z-11}) where $\psi=\hat \psi$. Let $\{T_k\}\subset (0,T)$
be such that $T_k\nearrow T$. Write $\varphi_k$ for the solution of
Equation (\ref{dualequation2-1}) with the initial condition
$\varphi_k(T)=\hat \varphi(T_k)$ (which belongs to $L^2(\Omega)$).
Let $\psi_k$ be given by (\ref{Y-Z-11}) where $\varphi=\varphi_k$.
Then, $\psi_k\in C([0,T];L^2(\Omega))$ solves (\ref{dualequation2}).
We claim that
\begin{equation}\label{Y-Z-13}
 \chi_\omega\psi_k \longrightarrow \chi_\omega\hat \psi \;\;\mbox{strongly  in }\;\;L^q(0,T; L^2(\omega)).
\end{equation}
When (\ref{Y-Z-13}) is proved, we get from Lemma~\ref{wanglemma4.3}
that $\chi_\omega\hat \psi\in Y_{T,q}$, which leads to (\ref{YZ-1}).

The remainder is to show (\ref{Y-Z-13}). Clearly, (\ref{Y-Z-13}) is
equivalent to \begin{equation}\label{Y-Z-16}
 \chi_\omega\varphi_k \longrightarrow \chi_\omega\hat \varphi\;\; \mbox{strongly in }\;\; L^q(0,T;L^2(\omega)).
\end{equation}
Let $\widetilde\varphi$ satisfy
\begin{eqnarray}\label{dualequation1}
\left\{
   \begin{array}{lll}
   \partial_t \widetilde\varphi(x,t)+\Delta \widetilde\varphi(x,t) -a_1(x)\widetilde\varphi(x,t)=0  &\mbox{ in }  &\Omega\times(-T,T),\\
   \widetilde\varphi(x,t)=0   &\mbox{ on }   &\partial\Omega\times(-T,T)
   \end{array}
\right.
\end{eqnarray}
and
 \begin{equation}\label{XT4}
  \widetilde\varphi(x,t)=\hat \varphi(x,t),\; (x,t)\in \Omega\times(0,T).
 \end{equation}
It is clear that
\begin{equation}\label{XT6}
 \widetilde\varphi\in C\big([-T,T);L^2(\Omega)\big)\cap L^q\big(-T,T;L^2(\omega)\big).
\end{equation}
 Because the equations satisfied by $\widetilde\varphi$ and
 $\varphi_k$ are time-invariant, one can easily check that
\begin{equation}\label{varphi-k}
 \varphi_k(t)= \widetilde\varphi(t-(T-T_k)),\; \mbox{ when } t\in(0,T).
\end{equation}
By (\ref{XT6}), we see that given $\varepsilon>0$,  there are two
positive constants $\delta(\varepsilon)$ and
$\eta(\varepsilon)=\eta(\varepsilon,\delta(\varepsilon))$ such that
\begin{equation}\label{XT1}
 \|\chi_\omega\widetilde\varphi\|_{L^q(a,b;L^2(\omega))} \leq \varepsilon,\;\;\mbox{when}\;\;(a,b)\subset(-T,T),  |a-b|\leq
 \delta(\varepsilon)
\end{equation}
and
\begin{equation}\label{XT5}
 \|\widetilde\varphi(a)-\widetilde\varphi(b)\|\leq \varepsilon, \;\;\mbox{when}\;\;(a,b)\subset\big[-T,T-\delta(\varepsilon)\big],  |a-b|\leq \eta(\varepsilon).
\end{equation}
Let $k_0=k_0(\varepsilon)$ verify that
\begin{equation}\label{klarge}
 0<T-T_k \leq \eta(\varepsilon),\;\;\mbox{when}\;\;k\geq k_0.
\end{equation}
From (\ref{XT4}) and (\ref{varphi-k}), it follows that
\begin{eqnarray*}
  & &\|\chi_\omega\varphi_k - \chi_\omega\hat\varphi \|_{L^q(0,T;L^2(\omega))}
  =\|\chi_\omega \big(\widetilde\varphi(\cdot-(T-T_k))-\widetilde\varphi(\cdot)\big)\|_{L^q(0,T;L^2(\omega))} \nonumber\\
    &\leq&  \|\chi_\omega \big(\widetilde\varphi(\cdot-(T-T_k))-\widetilde\varphi(\cdot)\big)\|_{L^q(0,T-\delta(\varepsilon);L^2(\omega))} \nonumber\\
    & & +  \|\chi_\omega \widetilde\varphi(\cdot-(T-T_k))\|_{L^q(T-\delta(\varepsilon),T;L^2(\omega))}
     +  \|\chi_\omega \widetilde\varphi\|_{L^q(T-\delta(\varepsilon),T;L^2(\omega))}.
\end{eqnarray*}
This, along with   (\ref{XT5}), (\ref{klarge}) and (\ref{XT1}),
yields that
\begin{equation*}
 \|\chi_\omega\varphi_k-\chi_\omega\hat\varphi \|_{L^q(0,T;L^2(\omega))} \leq (T-\delta(\varepsilon))^{\frac{1}{q}}
  \, \varepsilon + 2 \varepsilon\leq
  (T^{\frac{1}{q}}+2)\varepsilon,\;\;\mbox{when}\;\; k\geq k_0,
\end{equation*}
which leads to (\ref{Y-Z-16}), as well as (\ref{Y-Z-13}). This
completes the proof.
\end{proof}

\begin{Remark}\label{wangremark4.1}
The idea to show (\ref{Y-Z-16}) in the above proof is borrowed from
\cite{WZ} (see the proof of (3.8) on pages 2955-2957 in \cite{WZ}).
\end{Remark}

By  Theorem \ref{theorem-bang-bang} and
Proposition~\ref{lemma-YT-ZT}, we have the following consequence:

\begin{Corollary}\label{theorem-bang}
Let $y_0\in L^2(\Omega)\backslash\{0\}$. Suppose that  $a\in L^\infty(\Omega\times
\mathbb{R}^+)$ verifies $a(x,t)=a_1(x)+a_2(t)$ in
$\Omega\times\mathbb R^+$, with $a_1\in L^\infty(\Omega)$ and
$a_2\in L^\infty(\mathbb R^+)$. Then $(TP)_{y_0}^{M,p}$ has the
bang-bang property if and only if  $M\in (\widehat N_p(y_0),
\infty)$, where $\widehat N_p(y_0)$ is given by (\ref{HUANG1.12}).
\end{Corollary}

\section{Appendix}

\begin{proof}[The proof of (\ref{beta-t-T-C1}).] By the observability estimate for heat equations (see \cite[Proposition 3.2]{FZ}) and by (\ref{beta-t-T}), we have
\begin{equation}\label{Appendix-beta}
\beta(t,T)\leq
\mathrm{exp}\Big[C_0\Big(1+\frac{1}{T-t}+(T-t)+\big((T-t)^{\frac{1}{2}}+(T-t)\big)\|a\|_\infty+\|a\|_\infty^{\frac{2}{3}}\Big)\Big],
\end{equation}
where $C_0=C_0(\Omega,\omega)>0$ depends only on $\Omega$ and
$\omega$. Let  $y_0\in L^2(\Omega)\backslash\{0\}$. Define
$$
N_{\infty}(T,t,y_0)\triangleq
\inf\Big\{\|u\|_{L^\infty(t,T;L^2(\Omega))}\,\big|\,\Phi(T,t)y_0+\int_t^T\Phi(T,s)\chi_\omega
u(s)\,\mathrm{d}s=0\Big\},\; T>0, t\in [0,T).
$$
Here $\{\Phi(t,s)\; |\; 0\leq s\leq t<+\infty\}$ is the evolution
system generated by $\Delta- aI$ (see Chapter 5  in \cite{P}). By
the same way to prove Lemma~\ref{NP-YT}, we can obtain
\begin{equation}\label{Appendix-N}
   N_\infty(T,t,y_0) = \sup_{z\in L^2(\Omega)\backslash\{0\}} \frac{\langle \Phi(T,t)y_0,z \rangle
  }{\|\chi_w\Phi(T,\cdot)^*z\|_{L^1(t,T;L^2(\Omega))}}.
\end{equation}
From (\ref{beta-t-T}) and (\ref{Appendix-N}), it follows that
\begin{eqnarray}\label{Appendix-beta2}
  &&\nonumber\beta(t,T)= \sup_{z\in L^2(\Omega)\backslash\{0\}}\frac{\|\varphi(t;T,z)\| }{\|\chi_\omega\varphi(\cdot;T,z)\|_{L^1(t,T; L^2(\Omega))}}\\
   \nonumber&=& \sup_{z\in L^2(\Omega)\backslash\{0\}}\sup_{y_0\in L^2(\Omega)\backslash\{0\}}\frac{\langle\varphi(t;T,z),y_0\rangle }
   {\|\chi_\omega\varphi(\cdot;T,z)\|_{L^1(t,T; L^2(\Omega))}\cdot\|y_0\|} \\
  \nonumber&=&\sup_{y_0\in L^2(\Omega)\backslash\{0\}}\sup_{z\in L^2(\Omega)\backslash\{0\}}\frac{\langle\varphi(t;T,z),y_0\rangle }
  {\|\chi_\omega\varphi(\cdot;T,z)\|_{L^1(t, T; L^2(\Omega))}\cdot\|y_0\|} \\
   &=& \sup_{y_0\in L^2(\Omega)\backslash\{0\}}\frac{N_\infty(T,t,y_0)}{\|y_0\|}.
\end{eqnarray}
By the same way to show the monotonicity of $N_p(\cdot, y_0)$ (see
the proof of the part $(i)$ of Lemma \ref{lemma-N-t-lim}), we can
verify that for each $t\geq 0$ and $y_0\in
L^2(\Omega)\backslash\{0\}$, $N_\infty(\cdot,t,y_0)$ is
monotonically decreasing over $(t, \infty)$. This, along with
(\ref{Appendix-beta2}), yields that when $t\geq0$, $\beta(t,\cdot)$
is monotonically decreasing on $(t, \infty)$. When $(T-t)<1$,
(\ref{beta-t-T-C1}) follows from (\ref{Appendix-beta}) directly.
When $(T-t)\geq1$, we have $T\geq t+1$. By the monotonicity of
$\beta(t,\cdot)$, we have $\beta(t,T)\leq \beta(t,t+1)$. This, along
with (\ref{Appendix-beta}), yields $\beta(t,T)\leq
\mathrm{exp}(\widehat C_0)$, where $\widehat C_0$ depends only on
$\Omega$, $\omega$ and $\|a\|_\infty$. Hence, (\ref{beta-t-T-C1})
holds. This completes the proof.
\end{proof}

 \end{document}